\documentclass[12pt]{article}

\usepackage{amsmath, amsthm, amssymb, accents}
\usepackage{enumitem}
\usepackage{pdflscape}
\usepackage{caption}

\usepackage{ifpdf}
\ifpdf
\usepackage[pdftex]{graphicx}
\else
\usepackage[dvips]{graphicx}
\fi
\usepackage{tikz}
 	 \usetikzlibrary{arrows,backgrounds}
\usepackage[all]{xy}

\usepackage{multicol}

\input xy
\xyoption{all}

\usepackage[pdftex,plainpages=false,hypertexnames=false,pdfpagelabels]{hyperref}
\newcommand{\arxiv}[1]{\href{http://arxiv.org/abs/#1}{\tt arXiv:\nolinkurl{#1}}}
\newcommand{\arXiv}[1]{\href{http://arxiv.org/abs/#1}{\tt arXiv:\nolinkurl{#1}}}
\newcommand{\doi}[1]{\href{http://dx.doi.org/#1}{{\tt DOI:#1}}}

\newcommand{\googlebooks}[1]{(preview at \href{http://books.google.com/books?id=#1}{google books})}

\usepackage{xcolor}
\definecolor{dark-red}{rgb}{0.7,0.25,0.25}
\definecolor{dark-blue}{rgb}{0.15,0.15,0.55}
\definecolor{medium-blue}{rgb}{0,0,.8}
\definecolor{DarkGreen}{RGB}{0,150,0}
\definecolor{rho}{named}{red}
\hypersetup{
   colorlinks, linkcolor={purple},
   citecolor={medium-blue}, urlcolor={medium-blue}
}

\usepackage{longtable}
\usepackage{fullpage}

\theoremstyle{plain}
\newtheorem{thm}{Theorem}[section]
\newtheorem*{thm*}{Theorem}

\newtheorem{cor}[thm]{Corollary}

\newtheorem*{cor*}{Corollary}
\newtheorem{conj}[thm]{Conjecture}

\newtheorem*{conj*}{Conjecture}
\newtheorem{lem}[thm]{Lemma}
\newtheorem{prop}[thm]{Proposition}

\newtheorem*{quest*}{Question}
\newtheorem*{claim*}{Claim}

\theoremstyle{definition}

\newtheorem{construction}[thm]{Construction}
\newtheorem{defn}[thm]{Definition}

\newtheorem{rem}[thm]{Remark}

\DeclareMathOperator{\coev}{coev}

\DeclareMathOperator{\End}{End}
\DeclareMathOperator{\ev}{ev}
\DeclareMathOperator{\Forget}{Forget}
\DeclareMathOperator{\Hom}{Hom}

\DeclareMathOperator{\id}{id}
\DeclareMathOperator{\Irr}{Irr}
\DeclareMathOperator{\Tr}{Tr}

\DeclareMathOperator{\rev}{rev}

\newcommand{\comment}[1]{}

\newcommand{\be}{\begin{enumerate}[label=(\arabic*)]}
\newcommand{\ee}{\end{enumerate}}

\newcommand{\cZ}{\mathcal{Z}}
\newcommand{\cD}{\mathcal{D}}
\newcommand{\cE}{\mathcal{E}}
\newcommand{\cM}{\mathcal{M}}

\newcommand{\cC}{\mathcal{C}}

\def\semicolon{;}
\def\applytolist#1{
    \expandafter\def\csname multi#1\endcsname##1{
        \def\multiack{##1}\ifx\multiack\semicolon
            \def\next{\relax}
        \else
            \csname #1\endcsname{##1}
            \def\next{\csname multi#1\endcsname}
        \fi
        \next}
    \csname multi#1\endcsname}

\def\calc#1{\expandafter\def\csname c#1\endcsname{{\mathcal #1}}}
\applytolist{calc}QWERTYUIOPLKJHGFDSAZXCVBNM;
\def\bbc#1{\expandafter\def\csname bb#1\endcsname{{\mathbb #1}}}
\applytolist{bbc}QWERTYUIOPLKJHGFDSAZXCVBNM;
\def\bfc#1{\expandafter\def\csname bf#1\endcsname{{\mathbf #1}}}
\applytolist{bfc}QWERTYUIOPLKJHGFDSAZXCVBNM;
\def\sfc#1{\expandafter\def\csname s#1\endcsname{{\sf #1}}}
\applytolist{sfc}QWERTYUIOPLKJHGFDSAZXCVBNM;

\newcommand{\Rep}{{\sf Rep}}

\newcommand{\Bim}{{\sf Bim}}

\renewcommand{\Vec}{{\sf Vec}}
\newcommand{\Hilb}{{\sf Hilb}}

\newcommand{\fdHilb}{{\sf Hilb_{fd}}}
\newcommand{\Tube}{{\sf Tube}}

\newcommand{\noshow}[1]{}
\newcommand{\MR}[1]{}

\usetikzlibrary{shapes}
\usetikzlibrary{cd}
\usetikzlibrary{backgrounds}
\usetikzlibrary{decorations,decorations.pathreplacing,decorations.markings}
\usetikzlibrary{fit,calc,through}
\usetikzlibrary{external}
\tikzset{
	super thick/.style={line width=3pt}
}
\tikzstyle{knot}=[preaction={super thick, white, draw}]
\tikzstyle{shaded}=[fill=red!10!blue!20!gray!30!white]
\tikzstyle{unshaded}=[fill=white]
\tikzstyle{empty box}=[circle, draw, thick, fill=white, opaque, inner sep=2mm]
\tikzstyle{annular}=[scale=.7, inner sep=1mm, baseline]
\tikzstyle{rectangular}=[scale=.75, inner sep=1mm, baseline=-.1cm]
\tikzstyle{mid>}=[decoration={markings, mark=at position 0.5 with {\arrow{>}}}, postaction={decorate}]
\tikzstyle{mid<}=[decoration={markings, mark=at position 0.5 with {\arrow{<}}}, postaction={decorate}]
\tikzstyle{over}=[double, draw=white, super thick, double=]

\newcommand{\roundNbox}[6]{
	\draw[rounded corners=5pt, very thick, #1] ($#2+(-#3,-#3)+(-#4,0)$) rectangle ($#2+(#3,#3)+(#5,0)$);
	\coordinate (ZZa) at ($#2+(-#4,0)$);
	\coordinate (ZZb) at ($#2+(#5,0)$);
	\node at ($1/2*(ZZa)+1/2*(ZZb)$) {#6};
}

\newcommand{\ncircle}[5]{
	\draw[very thick, #1] #2 circle (#3);
	\node at #2 {#5};
	\filldraw[red] ($#2+(#4:#3cm)$) circle (.05cm);
}

\newcommand{\tikzmath}[2][]
     {\vcenter{\hbox{\begin{tikzpicture}[#1]#2
                     \end{tikzpicture}}}
     }

\newcommand{\tensor}[6]{
	\draw[rounded corners=5pt, very thick, unshaded] ($ #1 - (#2,#2) $) rectangle ($ #1 + (#2,#2) $);
	\draw ($ #1 + 14/23*(0,#2) - 1/3*(#2,0) $) -- ($ #1 - 14/23*(0,#2) - 1/3*(#2,0) $);
	\draw ($ #1 + 14/23*(0,#2) + 1/3*(#2,0) $) -- ($ #1 - 14/23*(0,#2) + 1/3*(#2,0) $);
	\draw[thick, red] ($ #1 - 1/3*(#2,0) - 1/5*(#2,0) $) -- ($ #1 - 5/6*(#2,0) $);
	\draw[thick, red] ($ #1 + 1/3*(#2,0) - 1/5*(#2,0) $) .. controls ++(180:.2cm) and ++(0:.2cm) .. ($ #1 - 1/3*(#2,0) + 2/5*(0,#2) $) .. controls ++(180:.2cm) and ++(0:.2cm) .. ($ #1 - 5/6*(#2,0) $);
	\draw[very thick] #1 ellipse ( {5/6*#2} and {2/3*#2});
	\filldraw[very thick, unshaded] ($ #1 + 1/3*(#2,0) $) circle (1/5*#2);
	\filldraw[very thick, unshaded] ($ #1 - 1/3*(#2,0) $) circle (1/5*#2);
	\node at ($ #1 + (.2,0) - 1/3*(#2,0) - 1/3*(0,#2) $) {\scriptsize{$#3$}};
	\node at ($ #1 + (.2,0) - 1/3*(#2,0) + 1/3*(0,#2) $) {\scriptsize{$#4$}};
	\node at ($ #1 + (.2,0) + 1/3*(#2,0) - 1/3*(0,#2) $) {\scriptsize{$#5$}};
	\node at ($ #1 + (.2,0) + 1/3*(#2,0) + 1/3*(0,#2) $) {\scriptsize{$#6$}};
}

\newcommand{\RevTensor}[6]{
	\draw[rounded corners=5pt, very thick, unshaded] ($ #1 - (#2,#2) $) rectangle ($ #1 + (#2,#2) $);
	\draw ($ #1 + 14/23*(0,#2) - 1/3*(#2,0) $) -- ($ #1 - 14/23*(0,#2) - 1/3*(#2,0) $);
	\draw ($ #1 + 14/23*(0,#2) + 1/3*(#2,0) $) -- ($ #1 - 14/23*(0,#2) + 1/3*(#2,0) $);
	\draw[thick, red] ($ #1 - 1/3*(#2,0) - 1/5*(#2,0) $) -- ($ #1 - 5/6*(#2,0) $);
	\draw[thick, red] ($ #1 + 1/3*(#2,0) - 1/5*(#2,0) $) .. controls ++(180:.2cm) and ++(0:.2cm) .. ($ #1 - 1/3*(#2,0) - 2/5*(0,#2) $) .. controls ++(180:.2cm) and ++(0:.2cm) .. ($ #1 - 5/6*(#2,0) $);
	\draw[very thick] #1 ellipse ( {5/6*#2} and {2/3*#2});
	\filldraw[very thick, unshaded] ($ #1 + 1/3*(#2,0) $) circle (1/5*#2);
	\filldraw[very thick, unshaded] ($ #1 - 1/3*(#2,0) $) circle (1/5*#2);
	\node at ($ #1 + (.2,0) - 1/3*(#2,0) - 1/3*(0,#2) $) {\scriptsize{$#3$}};
	\node at ($ #1 + (.2,0) - 1/3*(#2,0) + 1/3*(0,#2) $) {\scriptsize{$#4$}};
	\node at ($ #1 + (.2,0) + 1/3*(#2,0) - 1/3*(0,#2) $) {\scriptsize{$#5$}};
	\node at ($ #1 + (.2,0) + 1/3*(#2,0) + 1/3*(0,#2) $) {\scriptsize{$#6$}};
}

\newcommand{\PsiTensor}[6]{
	\draw[rounded corners=5pt, very thick, unshaded] ($ #1 - (#2,#2) - (.2,0) $) rectangle ($ #1 + (#2,#2) $);
	\draw ($ #1 + 14/23*(0,#2) - 1/3*(#2,0) $) -- ($ #1 - 14/23*(0,#2) - 1/3*(#2,0) $);
	\draw ($ #1 + 14/23*(0,#2) + 1/3*(#2,0) $) -- ($ #1 - 14/23*(0,#2) + 1/3*(#2,0) $);
	\draw[thick, red] ($ #1 - 1/3*(#2,0) - 1/5*(#2,0) $) -- ($ #1 - 5/6*(#2,0) $);
	\draw[thick, red] ($ #1 + 1/3*(#2,0) - 1/5*(#2,0) $) .. controls ++(180:.2cm) and ++(0:.2cm) .. ($ #1 - 1/3*(#2,0) + 2/5*(0,#2) $) .. controls ++(180:.2cm) and ++(0:.2cm) .. ($ #1 - 5/6*(#2,0) $);
	\draw[very thick] #1 ellipse ( {5/6*#2} and {2/3*#2});
	\filldraw[very thick, unshaded] ($ #1 + 1/3*(#2,0) $) circle (1/5*#2);
	\filldraw[very thick, unshaded] ($ #1 - 1/3*(#2,0) $) circle (1/5*#2);
	\node at ($ #1 + (.2,0) - 1/3*(#2,0) - 1/3*(0,#2) $) {\scriptsize{$#3$}};
	\node at ($ #1 + (.2,0) - 1/3*(#2,0) + 1/3*(0,#2) $) {\scriptsize{$#4$}};
	\node at ($ #1 + (.2,0) + 1/3*(#2,0) - 1/3*(0,#2) $) {\scriptsize{$#5$}};
	\node at ($ #1 + (.2,0) + 1/3*(#2,0) + 1/3*(0,#2) $) {\scriptsize{$#6$}};
  \node at ($ #1 - (#2,0) $) {$\Psi$};
}

\newcommand{\PsiRevTensor}[6]{
	\draw[rounded corners=5pt, very thick, unshaded] ($ #1 - (#2,#2) - (.2,0) $) rectangle ($ #1 + (#2,#2) $);
	\draw ($ #1 + 14/23*(0,#2) - 1/3*(#2,0) $) -- ($ #1 - 14/23*(0,#2) - 1/3*(#2,0) $);
	\draw ($ #1 + 14/23*(0,#2) + 1/3*(#2,0) $) -- ($ #1 - 14/23*(0,#2) + 1/3*(#2,0) $);
	\draw[thick, red] ($ #1 - 1/3*(#2,0) - 1/5*(#2,0) $) -- ($ #1 - 5/6*(#2,0) $);
	\draw[thick, red] ($ #1 + 1/3*(#2,0) - 1/5*(#2,0) $) .. controls ++(180:.2cm) and ++(0:.2cm) .. ($ #1 - 1/3*(#2,0) - 2/5*(0,#2) $) .. controls ++(180:.2cm) and ++(0:.2cm) .. ($ #1 - 5/6*(#2,0) $);
	\draw[very thick] #1 ellipse ( {5/6*#2} and {2/3*#2});
	\filldraw[very thick, unshaded] ($ #1 + 1/3*(#2,0) $) circle (1/5*#2);
	\filldraw[very thick, unshaded] ($ #1 - 1/3*(#2,0) $) circle (1/5*#2);
	\node at ($ #1 + (.2,0) - 1/3*(#2,0) - 1/3*(0,#2) $) {\scriptsize{$#3$}};
	\node at ($ #1 + (.2,0) - 1/3*(#2,0) + 1/3*(0,#2) $) {\scriptsize{$#4$}};
	\node at ($ #1 + (.2,0) + 1/3*(#2,0) - 1/3*(0,#2) $) {\scriptsize{$#5$}};
	\node at ($ #1 + (.2,0) + 1/3*(#2,0) + 1/3*(0,#2) $) {\scriptsize{$#6$}};
  \node at ($ #1 - (#2,0) $) {$\Psi$};
}

\newcommand{\evaluationMap}[3]{
	\draw[rounded corners=5pt, very thick, unshaded] ($ #1 - (#2,#2) - (.4,0) $) rectangle ($ #1 + (#2,#2) $);
	\draw ($ #1 - 2/3*(#2,0) - 1/2*(0,#2) $) .. controls ++(90:{#2}) and ++(90:{#2}) .. ($ #1 + 2/3*(#2,0) - 1/2*(0,#2) $);
	\draw[very thick] #1 circle ({5/6*#2});
	\filldraw[red] ($ #1 - 5/6*(#2,0) $) circle ({1/10*#2});
	\node at ($ #1 - (0,.1) $) {\scriptsize{$#3$}};
  \node at ($ #1 - (#2,0) -(.2,0) $) {$\Psi$};
}

\newcommand{\coevaluationMap}[3]{
	\draw[rounded corners=5pt, very thick, unshaded] ($ #1 - (#2,#2) - (.4,0) $) rectangle ($ #1 + (#2,#2) $);
	\draw ($ #1 - 2/3*(#2,0) + 1/2*(0,#2) $) .. controls ++(270:{#2}) and ++(270:{#2}) .. ($ #1 + 2/3*(#2,0) + 1/2*(0,#2) $);
	\draw[very thick] #1 circle ({5/6*#2});
	\filldraw[red] ($ #1 - 5/6*(#2,0) $) circle ({1/10*#2});
	\node at ($ #1 + (0,.1) $) {\scriptsize{$#3$}};
  \node at ($ #1 - (#2,0) -(.2,0) $) {$\Psi$};
}

\newcommand{\identityMap}[3]{
	\draw[rounded corners=5pt, very thick, unshaded] ($ #1 - (#2,#2) - (.4,0) $) rectangle ($ #1 + (#2,#2) $);
	\draw ($ #1 + 5/6*(0,#2) $) -- ($ #1 - 5/6*(0,#2) $);
	\draw[very thick] #1 circle ({5/6*#2});
	\filldraw[red] ($ #1 - 5/6*(#2,0) $) circle ({1/10*#2});
	\node at ($ #1 + (.2,0) $) {\scriptsize{$#3$}};
  \node at ($ #1 - (#2,0) -(.2,0) $) {$\Psi$};
}

\begin{document}
\title{Classification of finite depth objects in bicommutant categories via anchored planar algebras}
\author{Andr\'e Henriques, David Penneys, James Tener}
\date{}
\maketitle
\begin{abstract}
In our article [arXiv:1511.05226], we studied the commutant $\mathcal{C}'\subset \operatorname{Bim}(R)$ of a unitary fusion category $\mathcal{C}$, where $R$ is a hyperfinite factor of type $\rm II_1$, $\rm II_\infty$, or $\rm III_1$, and showed that it is a bicommutant category.
In other recent work [arXiv:1607.06041, arXiv:2301.11114] we introduced the notion of a (unitary) anchored planar algebra in a (unitary) braided pivotal category $\mathcal{D}$, and showed that they classify (unitary) module tensor categories for $\mathcal{D}$ equipped with a distinguished object.

Here, we connect these two notions and show that finite depth objects of $\mathcal{C}'$ are classified by connected finite depth unitary anchored planar algebras in $\mathcal{Z}(\mathcal{C})$.
This extends the classification of finite depth objects of $\operatorname{Bim}(R)$ by connected finite depth unitary planar algebras.
\end{abstract}

\tableofcontents

\section{Introduction}

Let $R$ be a hyperfinite factor which is either of type $\mathrm{II}_1$, $\mathrm{II}_\infty$, or $\mathrm{III}_1$.
And let $X\in\Bim(R)$ be a finite index (=dualizable) self-dual bimodule.
Recall that $X$ is said to have \emph{finite depth} if the subcategory of $\Bim(R)$ that it generates is a fusion category:
the object $X$ has finite depth iff
the total numer of isomorphism classes of irreducible $R$-$R$-bimodules which appear as summands of $X^{\boxtimes n}$, $n\in\bbN$, is finite.
The bimodule $X$ is called \emph{symmetrically self-dual} if it comes equipped with a unitary isomorphism $r:X\stackrel{\scriptscriptstyle \cong}{\rightarrow} \bar X$, which is fixed under the involution
$r\mapsto \bar r^*:\Hom(X, \bar X)\,\tikz[line width = .17mm,yscale=1.05,xscale=-1.05]{\draw[line cap=round] (40:.11) arc (40:360-41:.11);
\path (40:.11) arc (180+20:180-20:.11) coordinate (x);
\draw[line cap=round] (x) arc (180-20:180+20:.11) arc (85:85+40:.11);}
$.

For $R$ as above,
there is a well-known correspondence 
(see \cite{MR1055708,MR1339767,MR3635673,MR4236062}, or \cite[\S 3.2]{MR4581741} for a review) 
between conjugacy classes of finite depth symmetrically self-dual bimodules, and isomorphism classes of unitary connected finite depth planar algebras:
\[
\left\{
\parbox{5.3cm}{
\centerline{Finite depth, symm. self-dual} \centerline{$R$-$R$-bimodules}
}
\!\left\}\left/\text{conj.}
\,\leftrightarrow\,\,\,
\left\{
\parbox{4.3cm}{
\centerline{Connected finite depth} \centerline{unitary planar algebras}
}
\right\}\right.\right.\right.
\!\!\bigg/\text{iso.}
\]
In \cite[Conj.~1.2]{2301.11114}, we conjectured that this correspondence extends to the a certain class of bicommutant categories:

\begin{conj}\label{conjIntro}
Let $\cB$ be a bicommutant category which admits an absorbing object whose endomorphism algebra is a hyperfinite $\mathrm{II}_\infty$ or $\mathrm{III}_1$ factor.
Assume that $Z(\cB)\cong \Hilb(\cV)$ for some unitary braided fusion category $\cV$.
Then there exists a natural bijective correspondence:
\[
\left\{
\parbox{4.6cm}{
\rm Finite depth objects of $\cB$
}
\!\left\}\left/\text{\rm conj.}
\,\leftrightarrow\,\,\,
\left\{
\parbox{5.6cm}{
\rm Connected finite depth unitary \centerline{\rm anchored planar algebras in $\cV$}
}
\right\}\right.\right.\right.
\!\!\bigg/\text{\rm iso.}
\]
(Objects $X, Y\in\cB$ are called conjugate if there exists an invertible object $U\in \cB$ such that $Y\cong UXU^{-1}$.)
\end{conj}


The goal of this article is to provide evidence for the above conjecture, by proving it for certain bicommutant categories associated to unitary fusion categories. 
Let $R$ be a hyperfinite factor of type $\mathrm{II}_1$, $\mathrm{III}_\infty$, or $\mathrm{III}_1$.
Let $\cC\subset \Bim(R)$ be a unitary fusion subcategory, embedded in $\Bim(R)$ via a fully faithful unitary tensor functor, and let $\cC'$ be its commutant category (see \cite{MR3663592,MR4581741}, or Definition~\ref{def: C'} below).
The main result of this paper states that
Conjecture~\ref{conjIntro} is true when $\cB = \cC'$.
In this case, $Z(\cC')\cong \Hilb(Z(\cC)^{\rev})$ (Theorem~\ref{thm:Z(C')=Z(C)rev}), so finite depth objects of $\cC'$ correspond to connected finite depth unitary anchored planar algebras in $Z(\cC)^{\rev}$.
As one step in our proof, we prove, in Corollary \ref{cor:FactorizeBimR}, that there is an equivalence $\Bim(R)\cong\cC' \boxtimes_{Z(\cC)}\cC$.

\paragraph{Acknowledgements.}

David Penneys was supported by NSF DMS grants 1500387/1655912, 1654159, and 2154389. 
James Tener was supported by Australian Research Council Discovery Project DP200100067, as
well as by the Max Planck Institute for Mathematics, Bonn, during the initial stages of this
work.

\section{Background}

In our companion paper \cite{2301.11114} we introduced the notion of a unitary anchored planar algebra, and we showed that there is a bijective correspondence between connected finite depth unitary anchored planar algebras and unitary module fusion categories equipped with the choice of a symmetrically self-dual object.

\subsection{Unitary anchored planar algebras and unitary module tensor categories}

We rapidly recall the notion of unitary anchored planar algebra from \cite{MR4528312,2301.11114}. 
Let $\cV$ be a unitary ribbon fusion category.
Its unique unitary spherical structure induces a twist denoted $\theta_v:v\to v$.

\begin{defn}
A unitary anchored planar algebra $(\cP,r,\psi)$ over $\cV$ consists of:
\begin{enumerate}[label=(\arabic*)]
\item 
A sequence of \emph{box objects} $\cP[n]$ of $\cV$ for each $n\in \bbN_{\geq 0}$ together with a map $Z(T)$ in $\cV$ from the tensor product of the input box objects to the output box object corresponding to each anchored planar tangle $T$.
For example:
$$
Z\left(
\begin{tikzpicture}[baseline =-.1cm]
\pgftransformyscale{-1}
	\coordinate (a) at (0,0);
	\coordinate (b) at ($ (a) + (1.4,1) $);
	\coordinate (c) at ($ (a) + (.6,-.6) $);
	\coordinate (d) at ($ (a) + (-.6,.6) $);
	\coordinate (e) at ($ (a) + (-.8,-.6) $);
	\ncircle{}{(a)}{1.6}{89}{}
	\draw[thick, red] (d) .. controls ++(225:1.2cm) and ++(275:2.6cm) .. ($ (a) + (89:1.6) $);
	\draw[thick, red] ($ (a) + (89:1.6) $) to[out=-60, in=40] (.9,-.3);
	\draw (60:1.6cm) arc (150:300:.4cm);
	\draw ($ (c) + (0,.4) $) arc (0:90:.8cm);
	\draw ($ (c) + (-.4,0) $) circle (.25cm);
	\draw ($ (d) + (0,.88) $) -- (d) -- ($ (d) + (-.88,0) $);
	\draw ($ (c) + (0,-.88) $) -- (c) -- ($ (c) + (.88,0) $);
	\draw (e) circle (.25cm);
	\ncircle{unshaded}{(d)}{.4}{235}{}
	\ncircle{unshaded}{(c)}{.4}{225+180}{}
	\node[blue] at (c) {\small 2};
	\node[blue] at (d) {\small 1};
\end{tikzpicture}
\right)
:
\cP[3]\otimes \cP[5] \to \cP[6].
$$
This data should satisfy: 
\begin{itemize}
\item
(identity) the identity anchored tangle acts as the identity morphism
\item
(composition) if $S$ and $T$ are composable anchored planar tangles at input disk $i$ of $S$, then
$Z(S\circ_i T)=Z(S)\circ_i Z(T)$.

\item
(anchor dependence) the following relations hold:
\begin{itemize}
\item
(braiding)\hspace{.7cm}
$
Z\left(
\begin{tikzpicture}[baseline = -.1cm]
	\draw (-.6,-.2) -- (-.6,1);
	\draw (0,1) -- (0,.6);
	\draw (.6,-.2) -- (.6,1);
	\draw[thick, red] (0,-.6) -- (0,-1);
	\draw[thick, red] (0,.2) arc (0:-90:.2cm) -- (-.6,0) arc (90:270:.4cm) -- (-.2,-.8) arc (90:0:.2cm);
	\roundNbox{}{(0,0)}{1}{.2}{.2}{}
	\roundNbox{unshaded}{(0,-.4)}{.2}{.6}{.6}{}
	\roundNbox{unshaded}{(0,.4)}{.2}{.2}{.2}{}
	\node at (-.8,.8) {\scriptsize{$i$}};
	\node at (-.2,.8) {\scriptsize{$j$}};
	\node at (.8,.8) {\scriptsize{$k$}};
	\fill[red] (0,.2) circle (.05)  (0,-.6) circle (.05)  (0,-1) circle (.05);
\end{tikzpicture}
\right)
=
Z
\left(
\begin{tikzpicture}[baseline = -.1cm]
	\draw (-.6,-.2) -- (-.6,1);
	\draw (0,1) -- (0,.6);
	\draw (.6,-.2) -- (.6,1);
	\draw[thick, red] (0,-.6) -- (0,-1);
	\draw[thick, red] (0,.2) arc (180:270:.2cm) -- (.6,0) arc (90:-90:.4cm) -- (.2,-.8) arc (90:180:.2cm);
	\roundNbox{}{(0,0)}{1}{.2}{.2}{}
	\roundNbox{unshaded}{(0,-.4)}{.2}{.6}{.6}{}
	\roundNbox{unshaded}{(0,.4)}{.2}{.2}{.2}{}
	\node at (-.8,.8) {\scriptsize{$i$}};
	\node at (-.2,.8) {\scriptsize{$j$}};
	\node at (.8,.8) {\scriptsize{$k$}};
	\fill[red] (0,.2) circle (.05)  (0,-.6) circle (.05)  (0,-1) circle (.05);
\end{tikzpicture}
\right)
\,
\circ \beta_{\cP[j],\cP[i+k]}
$
\item
(twist)\hspace{1.3cm}
$
Z\left(
\begin{tikzpicture}[baseline=-.1cm]
	\ncircle{unshaded}{(0,0)}{1}{270}{}
	\ncircle{unshaded}{(0,0)}{.3}{270}{}
	\draw (90:.3cm) -- (90:1cm);
	\draw[thick, red] (-90:.3cm) .. controls ++(270:.3cm) and ++(270:.5cm) .. (0:.5cm) .. controls ++(90:.8cm) and ++(90:.8cm) .. (180:.7cm) .. controls ++(270:.6cm) and ++(90:.4cm) .. (270:1cm);
	\node at (100:.8cm) {\scriptsize{$n$}};
\end{tikzpicture}
\right)
=\theta_{\cP[n]}
$.
\end{itemize}
(Here, an $n$ next to a string indicates $n$ parallel strings.)
\end{itemize}

\item
Real structures $r_n : \cP[n] \to \overline{\cP[n]}$ satisfying 
for every anchored planar tangle $T$ 
$$ 
\overline{Z(T)}
\circ 
(r\otimes \cdots \otimes r)
=
r\circ 
Z(\overline{T}),
$$
where $\overline{T}$ is the reflection of $T$.

\item
A morphism $\psi_\cP: \cP[0] \to 1_\cV$ satisfying
$
\tikzmath{
\draw (.3,0) --node[above]{$\scriptstyle \cP[0]$} (1,0);
\draw[dotted] (1.6,0) --node[above]{$\scriptstyle 1_\cV$} (2.1,0);
\draw[dotted] (-.8,0) --node[above]{$\scriptstyle 1_\cV$} (-.3,0);
\roundNbox{unshaded}{(0,0)}{.3}{0}{0}{}
\draw[very thick] (0,0) circle (.2cm);
\filldraw[red] (-.2,0) circle (.05cm);
\roundNbox{unshaded}{(1.3,0)}{.3}{0}{0}{$\psi_\cP$}
}
=
\id_{1_\cV}
$, and
for every $n\geq 0$, we have an equality 
$$
\tikzmath{
	\coordinate (a) at (0,0);
	\pgfmathsetmacro{\boxWidth}{1};
	\draw (-3.2,-.5) -- (-1,-.5);
	\draw (-3.2,.5) -- (-1,.5);
	\draw (1,0) -- (2,0);
	\draw[dotted] (2,0) -- (3,0);
	\roundNbox{unshaded}{(-2.1,-.5)}{.3}{.1}{.1}{$r_n^{-1}$}
	\roundNbox{unshaded}{(2,0)}{.3}{0}{0}{$\psi_\cP$}
	\draw[rounded corners=5pt, very thick, unshaded] ($ (a) - (\boxWidth,\boxWidth) $) rectangle ($ (a) + (\boxWidth,\boxWidth) $);
	\draw ($ (a) + 1/3*(0,1) $) -- ($ (a) - 1/3*(0,\boxWidth) $);
	\draw[thick, red] ($ (a) + 1/3*(0,\boxWidth) - 1/5*(\boxWidth,0) $) -- ($ (a) - 2/3*(\boxWidth,0) $);
	\draw[thick, red] ($ (a) - 1/3*(0,\boxWidth) - 1/5*(\boxWidth,0) $) -- ($ (a) - 2/3*(\boxWidth,0) $);
	\draw[very thick] (a) ellipse ({2/3*\boxWidth} and {5/6*\boxWidth});
	\filldraw[very thick, unshaded] ($ (a) + 1/3*(0,\boxWidth) $) circle (1/5*\boxWidth);
	\filldraw[very thick, unshaded] ($ (a) - 1/3*(0,\boxWidth) $) circle (1/5*\boxWidth);
	\node at ($ (a) + (.2,0) $) {\scriptsize{$n$}};
	\node at (-2.85,-.3) {\scriptsize{$\overline{\cP[n]}$}};
	\node at (-1.35,-.3) {\scriptsize{$\cP[n]$}};
	\node at (-2,.7) {\scriptsize{$\cP[n]$}};
	\node at (1.35,.2) {\scriptsize{$\cP[0]$}};
	\node at (2.65,.2) {\scriptsize{$1_\cV$}};
}
=
\coev^\dag_{\cP[n]}.
$$
\end{enumerate}
\end{defn}

We now rapidly recall the notion of a unitary module fusion category from \cite{MR3578212,2301.11114}.
Let $\cV$ be as above.

\begin{defn}
A \emph{unitary module fusion category} $(\cC,\Phi^Z)$ consists of a unitary fusion category $\cC$ together with a pivotal braided unitary tensor functor $\Phi^Z: \cV\to Z(\cC)$.
\end{defn}

A \emph{pointing} of $(\cC,\Phi^Z)$ consists of a real object $c\in \cC$ such that $c$ and $\Phi^Z(\cV)$ generate $\cC$ under taking tensor product, orthogonal direct sum, and orthogonal direct summands.

Our main theorems in \cite{MR4528312,2301.11114} established an equivalence of categories
\[
\left\{
\parbox{5.6cm}{
\rm Connected finite depth unitary \centerline{\rm anchored planar algebras in $\cV$}
}
\left\}
\,\,\,\,\cong\,\,
\left\{\,\parbox{4.3cm}{\rm Pointed unitary module fusion\;\! categories over $\cV$}\,\right\}\right.\right.\!\!.
\]

\subsection{Bicommutant categories}
\label{sec:BicommutantCats}

The first author introduced the notion of bicommutant category in \cite{MR3747830}, and examples were constructed from unitary fusion categories and from conformal nets in \cite{MR3663592} and \cite{1701.02052}, respectively.
We refer the reader to \cite{MR1444286,MR3663592,MR3687214} for the basics of $\rm C^*$-tensor categories.

\begin{defn}
A \emph{bi-involutive} tensor category is a $\rm C^*$-tensor category $\cC$ equipped with a covariant anti-linear unitary functor $\overline{\,\cdot\,}:\cC\to \cC$ called the \emph{conjugate}.
There are coherence natural isomorphisms $\varphi_c: c\to \overline{\overline{c}}$, $\nu_{a,b}:\overline{a}\otimes \overline{b}\to \overline{b\otimes a}$ and $r: 1\to \overline{1}$ satisfying monoidal coherences (see \cite[Def.~2.3]{MR3663592}).
\end{defn}

Basic examples include $\Hilb$, $\Bim(R)$ for a von Neumann algebra $R$, and unitary tensor categories, a.k.a.~semisimple rigid $\rm C^*$-tensor categories with simple unit object.
We refer the reader to \cite[p.5]{MR3663592} or \cite[\S3.5]{MR4133163} for more details on this last example.

\begin{defn}
A \emph{bi-involutive tensor functor} $F: \cA\to \cB$ between bi-involutive tensor categories is a unitary tensor functor equipped with a unitary natural isomorphism $\chi_a : F(\overline{a})\to \overline{F(a)}$ for $a\in \cA$ satisfying monoidal and involutive coherences (see \cite[Def.~2.5]{MR3663592} or \cite[Def.~3.35]{MR4133163}).

A \emph{representation} of a bi-involutive category $\cC$ is a von Neumann algebra $R$ together with a bi-involutive tensor functor $\alpha: \cC\to \Bim(R)$.
\end{defn}

In \cite[\S5]{MR4581741}, we introduced a certain extra structure on $\Bim(R)$ which we called a \emph{positive structure}. 
These positive structures play no role in the present paper, as representations of unitary fusion categories uniquely extend to positive representations by \cite[Thm.~A]{MR4581741}, so we will not emphasize them here.

\begin{defn}\label{def: C'}
Suppose $(\cC,\alpha)$ is a bi-involutive tensor category equipped with a representation into $\Bim(R)$.
The \emph{commutant category} $\cC'$ is the unitary relative center of $\alpha(\cC)$ inside $\Bim(R)$ \cite[\S2.3]{MR3663592}.
It has:
\begin{itemize}
\item
objects bimodules $X\in \Bim(R)$ equipped with half-braidings $e_X=\{e_{X,c}: X\boxtimes \alpha(c)\to \alpha(c)\boxtimes X\}_{c\in \cC}$ satisfying the usual hexagon relation and naturality relation.
Our graphical convention for the half-braiding $e_X$ is
$$
e_{X,c}
=
\tikzmath{
\draw (.4,-.4) node[below]{$\scriptstyle \alpha(c)$} to[out=90, in=-90] (-.4,.4) node[above]{$\scriptstyle \alpha(c)$};
\draw[knot] (-.4,-.4) node[below]{$\scriptstyle X$} to[out=90, in=-90] (.4,.4) node[above]{$\scriptstyle X$};
}
:
X\boxtimes \alpha(c) \to \alpha(c)\boxtimes X.
$$

\item
morphisms $f: (X,e_X) \to (Y,e_Y)$ are bimodule maps $f: X\to Y$ which commute with the half-braidings, i.e., 
$$
\tikzmath{
\draw (.4,-1.3) node[below]{$\scriptstyle \alpha(c)$} -- (.4,.-.4) to[out=90, in=-90] (-.4,.4) node[above]{$\scriptstyle \alpha(c)$};
\draw[knot] (-.4,-.4) node[left, yshift=.2cm]{$\scriptstyle Y$}  to[out=90, in=-90] (.4,.4) node[above]{$\scriptstyle Y$};
\draw (-.4,-1.3) node[below]{$\scriptstyle X$} -- (-.4,-1);
\roundNbox{}{(-.4,-.7)}{.3}{0}{0}{$f$}
}
=
\tikzmath{
\draw (.4,-.4) node[below]{$\scriptstyle \alpha(c)$} to[out=90, in=-90] (-.4,.4) -- (-.4,1.3) node[above]{$\scriptstyle \alpha(c)$};
\draw[knot] (-.4,-.4) node[below]{$\scriptstyle X$}  to[out=90, in=-90] (.4,.4) node[right, yshift=-.2cm]{$\scriptstyle X$};
\draw (.4,1.3) node[above]{$\scriptstyle Y$} -- (.4,1);
\roundNbox{}{(.4,.7)}{.3}{0}{0}{$f$}
}\,.
$$
\item
tensor product is given by
$$
(X,e_X)\otimes (Y,e_Y)
:=
\left(X\boxtimes Y, e_{X\boxtimes Y}
\right)
\qquad\qquad
e_{X\boxtimes Y, c}
:=
\tikzmath{
\draw (.4,-.4) node[below]{$\scriptstyle \alpha(c)$} to[out=90, in=-90] (-.4,.4) to[out=90, in=-90] (-1.2,1.2) node[above]{$\scriptstyle \alpha(c)$};
\draw[knot] (-.4,-.4) node[below]{$\scriptstyle Y$} to[out=90, in=-90] (.4,.4) -- (.4,1.2) node[above]{$\scriptstyle X$};
\draw[knot] (-1.2,-.4) node[below]{$\scriptstyle X$} -- (-1.2,.4) to[out=90, in=-90] (-.4,1.2) node[above]{$\scriptstyle X$};
}\,.
$$
\item the bi-involutive structure is given in \cite[\S2.3]{MR3663592}.
\end{itemize}
\end{defn}

Observe that $\cC'$ comes equipped with a canonical representation $\alpha':\cC'\to \Bim(R)$ by forgetting the half-braiding.
It thus makes sense to talk about the \emph{bi-commutant} $\cC''$.
Observe there is a canonical bi-involutive tensor functor $\cC\to \cC''$ given by $c\mapsto (\alpha(c), \{e_{X,c}^{-1}\}_{(X,e_X)\in \cC'})$, where  $e_{X,c}: X\boxtimes \alpha(c)\to \alpha(c)\boxtimes X$ is the half-braiding for $X\in \cC'$.

\begin{defn}
Let $R$ be a von Neumann algebra.
A \emph{bicommutant category} in $\Bim(R)$ is a bi-involutive tensor category $\cB$ equipped with a representation $$\alpha: \cB\to \Bim(R)$$ such that the canonical map $\cB\to \cB''$ is an equivalence.
(The category $\cB$ is equipped with a positive structure, inherited from $\Bim(R)$.)
\end{defn}

\begin{thm}\label{thm:Z(C')=Z(C)rev}
Suppose that $\cB$ is a bicommutant category in $\Bim(R)$.
There is an equivalence of braided categories $Z(\cB)^{\rev}\to Z(\cB')$ making the following diagram commute:
\begin{equation}
\label{eq:Z(C')=Z(C)rev commutes}
\begin{tikzcd}
Z(\cB)^{\rev}
\arrow[dr,swap,"\operatorname{Canonical}"]
\arrow[rr,"\cong"]
&&
Z(\cB')
\arrow[dl,"\operatorname{Forget}"]
\\
& \cB'.
\end{tikzcd}
\end{equation}
\end{thm}
\begin{proof}
To construct a functor $Z(\cB)^{\rev}\to Z(\cB')$, an object $$\underline X:=(X,e_X=\{e_{X,Y}:X\otimes Y\to Y\otimes X\}_{Y\in \cC})$$ in $Z(\cB)^{\rev}$ maps to $$\underline {\underline X} = (\underline X,e_{\underline X}) = ((X,e_X),e_{\underline X})$$
in $Z(\cB')$, where for $\underline Y:=(Y,e_Y)\in\cB'$ we have $e_{\underline X,\underline Y}=e_{Y,X}^{-1}$.
This is evidently a strict tensor functor.
It is a braided tensor functor because the braiding on $Z(\cB^{\rev})$ is given by
$$
{\beta^{Z(\cB)^{\rev}}}_{\underline X,\underline Y}=
({\beta^{Z(\cB)}}_{\underline Y,\underline X})^{-1} = e_{Y,X}^{-1}
$$
which agrees with the braiding $Z(\cB')$:
$$
\beta^{\cB'}_{\underline{\underline X},\underline{\underline Y}}= e_{\underline X,\underline Y}=e_{Y,X}^{-1}.
$$
The diagram \eqref{eq:Z(C')=Z(C)rev commutes} visibly commutes. It remains to show the top arrow is an equivalence.

The same construction replacing $\cB$ with $\cB'$ gives a strict braided tensor functor $Z(\cB')^{\rev} \to Z(\cB'')$.
Hence we also get a strict braided tensor functor $Z(\cB')\to Z(\cB'')^{\rev}$ by taking the reverse braiding on each side.
The composite map $Z(\cB)^{\rev} \to Z(\cB')\to Z(\cB'')^{\rev}$ is the map induced from the equivalence $\cB\to \cB''$.
Again, replacing $\cB$ with $\cB'$, the composite map $Z(\cB')\to Z(\cB'')^{\rev} \to Z(\cB''')$ is also the map induced from the equivalence $\cB'\to \cB'''$.
This completes the proof.
\end{proof}

We now prove some useful results in the case that $\cC$ is a unitary fusion category fully faithfully embedded in $\Bim(R)$, for a von Neumann factor $R$ not of type $\rm I$.
In \cite{MR3663592}, we proved that both $\cC''=\Hilb(\cC)=\cC\otimes_{\fdHilb} \Hilb$ and $\cC'$ are bicommutant categories.

\begin{lem}
\label{lem:ZHilbC=HilbZC}
If $\cC$ is a unitary fusion category, then
$ Z(\Hilb(\cC))\cong \Hilb(Z(\cC))$.
\end{lem}
\begin{proof}
The canonical functor $\Hilb(Z(\cC)) \to Z_\cC(\Hilb(\cC))=Z(\Hilb(\cC))$ is visibly fully faithful, but it is not obvious that it is essentially surjective. 

The induction functor $I:\cC \to Z(\cC)$ (adjoint to the forgetful functor) is such that every object $(X,e_X)\in Z(\cC)$ is a direct summand of $I(X)$.
The same property holds true for the corresponding functor $\Hilb(I):\Hilb(\cC) \to \Hilb(Z(\cC))$.
Every $(X,e_X)\in Z(\Hilb(\cC))$ is a direct summand of $\Hilb(I)(X)\in \Hilb(Z(\cC))$, hence lives in $\Hilb(Z(\cC))$.
\end{proof}

\begin{prop}
The functors $Z(\cC)\to \cC'$ and $Z(\cC')\to \cC'$
are fully faithful.
\end{prop}
\begin{proof}
We first show that $Z(\cC)\to \cC'$ is fully faithful.
Let $(X,e_X), (Y,e_Y)\in Z(\cC)$.
Every morphism $(X,e_X) \to (Y,e_Y)$ in $\cC'$ is a morphism $X\to Y$ in $\cC\subset\Bim(R)$ compatible with the half-braidings.
This is exactly the definition of a morphism in $Z(\cC)$.

By Lemma \ref{lem:ZHilbC=HilbZC}, it follows that $\Hilb(Z(\cC))\cong Z(\Hilb(\cC))\to \cC'$ is also fully faithful.
Using that $\Hilb(\cC)$ is a bicommutant category,
the result now follows by commutativity of \eqref{eq:Z(C')=Z(C)rev commutes} for $\cB=\Hilb(\cC)$.
\end{proof}

We will need the following corollary in our construction later on.

\begin{cor}
\label{cor:M is a Z(C)rev module tensor cat}
Suppose $\cM\subset \cC'$ is a full tensor subcategory containing the image of $Z(\cC)$ in $\cC'$.
Then there is a fully faithful braided tensor functor $Z(\cC)^{\rev}\to Z(\cM)$ such that the following diagram commutes:
$$
\begin{tikzcd}
Z(\cC)^{\rev} 
\arrow[rr,hookrightarrow,""]
\arrow[dr]
&&
Z(\cM)
\arrow[dl]
\\
& \cM
\end{tikzcd}
$$
\end{cor}
\begin{proof}
The image of $Z(\cC)$ in $\cC'$ lifts to a commutative diagram
\[
\begin{tikzcd}
Z(\cC)^{\rev} 
\arrow[r,hookrightarrow,""]
\arrow[dr]
&
Z(\Hilb(\cC))^{\rev}
\arrow[r,"\cong"]
\arrow[d]
&
Z(\cC')
\arrow[dl]
\\
&\cC'
\end{tikzcd}
\]
in which the horizontal arrows are braided tensor functors.
The image of $Z(\cC)^{\rev}\hookrightarrow Z(\cC')$ lies in $Z(\cM)$
because a half-braiding with $\cC'$ resticts to a half-braiding with $\cM$.

Finally, the braided tensor functor $Z(\cC)^{\rev}\to Z(\cM)$ is automatically fully faithful as $Z(\cC)^{\rev}$ is modular \cite[Cor.~3.26]{MR3039775}.
\end{proof}

\section{Relative tensor product of module tensor categories}
\label{sec:BalancedTensor}

In this section, we analyze the relative tensor product $\cC\boxtimes_\cV \cD$ of a $\cV$-module tensor category $(\cD,\Phi^Z:\cV\to \cD)$ and a $\cV^{\rev}$-module tensor category $(\cC,\Psi^Z:\cV\to \cC)$.
Here, $\Psi^Z$ is a reverse-braided functor, meaning that it satisfies $\Psi^Z(\beta_{u,v})=\beta^{-1}_{\Psi^Z(v),\Psi^Z(u)}$ for $u,v\in\cV$.
We shall use the notational convention $\Phi := \Forget\circ \Phi^Z$ and $\Psi := \Forget\circ \Psi^Z$.
Throughout this section, we assume $\cV$ is semisimple with simple unit object.

The monoidal category $\cC\boxtimes_\cV \cD$ is defined via the universal property which states that for every tensor category $\cE$, the data of
a $\cV$-balanced tensor functor 
$B:\cC \boxtimes \cD \to \cE$
is equivalent to the data of a tensor functor
$
B':\cC\boxtimes_\cV\cD \to \cE,
$ 
via the commutative diagram
$$
\begin{tikzcd}
\cC\boxtimes \cD
\arrow[dr,"B"']
\arrow[rr,"-\boxtimes_\cV-"]
&&
\cC\boxtimes_\cV \cD
\arrow[dl,"B'"]
\\
&
\cE
\end{tikzcd}.
$$
Here, a tensor functor $\cC \boxtimes \cD \to \cE$ is called $\cV$-\emph{balanced} if it comes with monoidal natural isomorphisms 
$$
\eta_{c,v,d}:B\big((c \otimes\Psi v) \boxtimes d\big) \to B\big(c\boxtimes (\Phi v\otimes d\big))
$$
for $c\in \cC$, $d\in \cD$, and $v\in \cV$,
satisfying the coherence which states that passing $v_1\otimes v_2$ from one side to the other is the same as first passing $v_2$ and then passing $v_1$.
Note that the monoidal coherence involves the two half-braidings (the bottom vertical arrows):
\begin{equation}
\label{eq:MonoidalBalancingAxiom}
\begin{tikzcd}
B((c_1\otimes \Psi v_1)\boxtimes d_1) \otimes B((c_2 \otimes \Psi v_2)\boxtimes d_2)
\arrow[d]
\arrow[r,"\eta\otimes \eta"]
&
B(c_1\boxtimes (\Phi v_1\otimes d_1)) \otimes B(c_2 \boxtimes (\Phi v_2\otimes d_2))
\arrow[d]
\\
B((c_1\otimes \Psi v_1)\boxtimes d_1) \otimes ((c_2 \otimes \Psi v_2)\boxtimes d_2)
\arrow[d,equals]
&
B(c_1\boxtimes (\Phi v_1\otimes d_1)\otimes (c_2 \boxtimes (\Phi v_2\otimes d_2))
\arrow[d,equals]
\\
B((c_1\otimes \Psi v_1 \otimes c_2 \otimes \Psi v_2)\boxtimes (d_1\otimes d_2))
\arrow[d]
&
B((c_1\otimes c_2)\boxtimes (\Phi v_1\otimes d_1\otimes \Phi v_2\otimes d_2))
\arrow[d]
\\
B((c_1\otimes c_2 \otimes \Psi v_1 \otimes \Psi v_2)\boxtimes (d_1\otimes d_2))
\arrow[r, "\eta"]
&
B((c_1\otimes c_2)\boxtimes (\Phi v_1\otimes \Phi v_2\otimes d_1\otimes d_2))
\end{tikzcd}
\end{equation}
We write $\zeta_v$ for the isomorphism $\eta_{1,v,1}:\Phi v \boxtimes_\cV 1_\cD\to 1_\cC\boxtimes_\cV \Phi v$.

\begin{rem}\label{rem:1Composition}
When the braided category $\cV$ is fusion and when $\cC$ and $\cD$ are multifusion, then
the balanced tensor product $\cC\boxtimes_\cV \cD$ agrees with the composition of $1$-morphisms
$
\Vec \xrightarrow{\cC} \cV \xrightarrow{\cD} \Vec
:=
\cC\boxtimes_\cV \cD
$
in the $4$-category of braided fusion categories \cite{MR3650080,MR3590516,MR4228258,MR4498161}.
\end{rem}

\subsection{The ladder category model}

One model for $\cC\boxtimes_\cV \cD$ is the Cauchy completion of the \emph{ladder category} \cite{MR2978449,MR3975865}\footnote{The ladder category construction is the analog of the balanced tensor product for non-Cauchy complete categories.}, which we denote by $\cC\boxdot_\cV \cD$.
The tensor category $\cC\boxdot_\cV \cD$ is known to satisfy the universal property for $\cC\boxtimes_\cV\cD$ with respect to ordinary (non-monoidal) $\cV$-balanced functors from $\cC\boxtimes \cD$, and thus satisfies the universal property for $\cC\boxtimes_\cV\cD$ with respect to $\cV$-balanced \emph{tensor} functors from $\cC\boxtimes \cD$ by \cite[Prop 3.2.3.1.(4)]{LurieHigherAlgebra}.

\begin{defn}
Objects in the ladder category are of the form $a\boxdot d$ where $a\in \cC$ and $d\in \cD$, and
the morphism spaces are given by
$$
(\cC\boxdot_\cV \cD)\big(a\boxdot c\to b\boxdot d\big)
:=
\bigoplus_{c\in \Irr(\cV)} \cC\big(a\to b\otimes \Psi v\big) \otimes \cD\big(\Phi v\otimes c \to d\big).
$$
We refer the reader to \cite[Def.~7]{MR3975865} for the remainder of the definition as a linear category and to  \cite[\S3.2]{mitchell-thesis} for the monoidal structure.
\end{defn}

\begin{lem}
\label{lem:CanonicalInclusionsFF}
Consider the canonical functors
\begin{align*}
F: \cC\to \cC\boxtimes 1_\cD \subset \cC\boxtimes \cD &\xrightarrow{-\boxtimes_\cV-} \cC\boxtimes_\cV \cD
\\
G: \cD\to 1_\cC\boxtimes \cD \subset \cC\boxtimes \cD &\xrightarrow{-\boxtimes_\cV-} \cC\boxtimes_\cV \cD.
\end{align*}
\begin{enumerate}[label=(\arabic*)]
\item 
If $\Phi: \cV\to \cD$ is fully faithful
, then $F:\cC\to \cC\boxtimes_\cV \cD$ is fully faithful.
\item 
If $\Psi: \cV\to \cC$ is fully faithful
, then $G:\cD\to \cC\boxtimes_\cV \cD$ is fully faithful.
\end{enumerate}
\end{lem}
\begin{proof}
We only prove (1) as (2) is similar.
We use the ladder category model $\cC\boxdot_\cV \cD$ for $\cC\boxtimes_\cV \cD$.
Observe that
$$
(\cC\boxdot_\cV \cD)\big(F(a)\to F(b)\big)
=
(\cC\boxdot_\cV \cD)\big(a\boxdot 1_\cD\to b\boxdot 1_\cD\big)
=
\bigoplus_{v\in \Irr(\cV)} \cC\big(a\to b\otimes \Psi v\big) \otimes \cD\big(\Phi v \to 1_\cD\big).
$$
If $\Phi$ is fully faithful,
the only $v\in \Irr(\cV)$ with $\cD(\Phi v \to 1_\cD)\neq 0$ is $1_\cV$, and moreover, $\cD(\Phi 1_\cV\to 1_\cD)=\cD(\Phi 1_\cV\to \Phi 1_\cV)\cong \cV(1_\cV\to 1_\cV)=\bbC$.
The result follows.
\end{proof}

\begin{cor}
\label{cor:RelativeProductFusion}
Suppose $\cC,\cD,\cV$ are all fusion.
If $\Psi: \cV\to \cC$ or $\Phi: \cV\to \cD$ is fully faithful, then $\cC\boxtimes_\cV \cD$ is fusion.
\end{cor}
\begin{proof}
When $\cC,\cD,\cV$ are all fusion, then $\cC\boxtimes_\cV\cD$ is always multifusion as it is 1-composition in the 4-category of braided fusion categories (see Remark \ref{rem:1Composition}).
When $\Psi: \cV\to \cC$ or $\Phi: \cV\to \cD$ is fully faithful, then $\cC\boxtimes_\cV \cD$ has simple unit object by Lemma \ref{lem:CanonicalInclusionsFF}, and is thus fusion.
\end{proof}

\subsection{Canonical centralizing structure}

We now discuss the notion of \emph{centralizing structure} for two tensor categories $\cA,\cB$ equipped with tensor functors to some other tensor category $\cC$ -- see \cite[Def.~3.24]{2111.06378}.

\begin{defn}
Let $\cA,\cB,\cC$ be tensor categories, and suppose we have tensor functors $F: \cA\to \cC$ and $G: \cB\to \cC$.
A \emph{centralizing structure} 
is a family of natural isomorphisms
$\{\sigma_{a,b}: F(a)\otimes G(b) \to G(b)\otimes F(a)\}_{a\in\cA,b\in\cB}$ satisfying the following conditions, where coherence isomorphisms have been suppressed:
\begin{itemize}
\item 
For $a\in \cA$ and $b,b'\in \cB$,
$(\id_{G(b)}\otimes \sigma_{a,b'})\circ (\sigma_{a,b}\otimes \id_{G(b')}) = \sigma_{a,b\otimes b'}$
\item
For $a,a'\in \cA$ and $b\in \cB$,
$(\sigma_{a,b}\otimes \id_{F(a')})\circ (\id_{F(a)}\otimes \sigma_{a',b}) = \sigma_{a\otimes a',b}$.
\end{itemize}
\end{defn}

A centralizing structure for $F: \cA\to \cC$ and $G: \cB\to \cC$ is equivalent to the data needed to promote $F\boxtimes G : \cA\boxtimes \cB \to \cC$ to a tensor functor.

\begin{construction}
\label{construction:CanonicalCentralizing}
By universality of $\cC\boxtimes_\cV \cD$ for $\cV$-balanced functors out of $\cC\boxtimes \cD$,
$\cC\boxtimes_\cV \cD$ comes equipped with a canonical centralizing structure 
for the canonical functors from Lemma \ref{lem:CanonicalInclusionsFF}.
Indeed, since 
$$
(c\boxtimes 1)\otimes (1\boxtimes d)
=
c\boxtimes d
=
(1\boxtimes d)\otimes (c\boxtimes 1)
$$
in $\cC\boxtimes \cD$
for all $c\in \cC$ and $d\in \cD$,
we get a canonical isomorphism $\sigma_{c,d}: F(c)\otimes G(d) \to G(d)\otimes F(c)$ 
from the tensorator of $-\boxtimes_\cV - : \cC\boxtimes \cD\to \cC\boxtimes_\cV \cD$:
$$
(c\boxtimes_\cV 1)\otimes (1\boxtimes_\cV d) 
\xrightarrow{\cong} 
c\boxtimes_\cV d
\xrightarrow{\cong}
(1\boxtimes_\cV d) \otimes (c\boxtimes_\cV 1).
$$
The centralizing axioms are satisfied by associativity of the tensorator.

Moreover, we have the following additional property of this centralizing structure:
whenever $c$ or $d$ is in the image of $\cV$, the centralizing structure $\sigma$ is compatible with the half-braiding coming from the image of $\cV$ in $Z(\cC)$ or $Z(\cD)$ respectively.
For example, when $c=\Psi v$, the following diagram commutes:
\[
\begin{tikzcd}
(\Psi v \boxtimes_\cV 1)\otimes (1\boxtimes_\cV d)
\arrow[r,"\zeta_v \otimes \id"]
\arrow[dd, "\sigma_{\Psi v, d}"]
\arrow[dr, "\cong"]
&
(1\boxtimes_\cV \Phi v)\otimes (1\boxtimes_\cV d)
\arrow[r,"\cong"]
&
1 \boxtimes_\cV (\Phi v\otimes d)
\arrow[dd, "\id\boxtimes_\cV e_{\Phi v, d}"]
\\
&
\Psi v \boxtimes_\cV d
\arrow[dl, "\cong"]
\arrow[ur, "\eta_{1,v,d}"]
\\
(1\boxtimes_\cV d)\otimes (\Psi v \boxtimes_\cV 1)
\arrow[r,"\id\otimes \zeta_v"]
&
(1\boxtimes_\cV d)\otimes (1\boxtimes_\cV \Phi v)
\arrow[r,"\cong"]
&
1 \boxtimes_\cV (d\otimes \Phi v)
\end{tikzcd}
\]
The top square is naturality of $\eta$ (recall $\zeta_v=\eta_{1,v,1}$), the left triangle is the definition of $\sigma$, and the bottom right pentagon is \eqref{eq:MonoidalBalancingAxiom} for $B=-\boxtimes_\cV-$ setting
$c_1=c_2=1_\cC$, $v_1=1_\cV$ and $v_2=v$, and $d_1=d$ and $d_2=1_\cD$.
Similarly, when $d=\Phi v$, the following diagram commutes
\begin{equation}
\label{eq:SigmaCompatibility}
\begin{tikzcd}
(c\otimes \Psi v) \boxtimes 1
\arrow[r,"\cong"]
\arrow[dr, "\eta_{c,v,1}"]
&
(c\boxtimes_\cV 1)
\otimes
(\Psi v \boxtimes_\cV 1) 
\arrow[r,"\id\otimes \zeta_v "]
&
(c\boxtimes_\cV 1)
\otimes
(1 \boxtimes_\cV \Phi v) 
\arrow[dd, "\sigma_{c, \Phi v}"]
\\
&
c \boxtimes_\cV \Phi v
\arrow[ur, "\cong"]
\arrow[dr, "\cong"]
\\
(\Psi v\otimes c) \boxtimes 1
\arrow[r,"\cong"]
\arrow[uu, "e_{\Psi v, c}\boxtimes_\cV \id"]
&
(\Psi v \boxtimes_\cV 1) 
\otimes
(c\boxtimes_\cV 1)
\arrow[r,"\zeta_v\otimes \id"]
&
(1 \boxtimes_\cV \Phi v)\otimes (c\boxtimes_\cV 1).
\end{tikzcd}
\end{equation}
\end{construction}

\begin{lem}
\label{lem:RelativeCenterFF}
If $\Psi^Z: \cV\to Z(\cC)$ is fully faithful
, then the functor 
which maps $\cD$ to the relative commutant $Z_\cC(\cC\boxtimes_\cV \cD)$ of $\cC$ inside $\cC\boxtimes_\cV \cD$
is fully faithful.
\end{lem}
\begin{proof}
Using the ladder category model $\cC\boxdot_\cV \cD$, 
$$
(\cC\boxdot_\cV \cD)(1\boxdot d_1\to 1\boxdot d_2)
=
\bigoplus_{v\in \Irr(\cV)} \cC\big(1\to 1\otimes \Psi v\big) \otimes \cD\big(\Phi v \otimes d_1\to d_2\big).
$$
A morphism $f:1\boxdot d_1\to 1\boxdot d_2$ lies in $Z_\cC(\cC\boxdot_\cV \cD)$ exactly when it is compatible with the half-braidings, which are induced from the centralizing structure.
Writing $f=\sum_{v\in \Irr(\cV),i} g_{v,i}\otimes h_{v,i} $ as a sum of its $v$-components, where $g_{v,i}:1_\cC\to \Psi(v)$ and $h_{v,i}:\Phi(v)\otimes d_1\to d_2$, where the $h_{v,i}$ form a basis of $\cD\big(\Phi v \otimes d_1\to d_2\big)$,
this compatibility reduces via \eqref{eq:SigmaCompatibility} to
$$
e_{\Psi v,c}\circ (g_{v,i}\otimes \id_c) = \id_c\otimes g_{v,i} 
\qquad\qquad\qquad
\forall\, c\in \cC,\, \forall\, v\in \Irr(\cV).
$$
This implies that 
$g_{v,i} \in \cC(1\to \Psi v)$ is actually a morphism in $Z(\cC)$ and thus $g_{v,i} \in Z(\cC)(1\to \Psi^Z v)$.
Since $\Psi^Z$ was assumed to be fully faithful, 
$g_{v,i}=0$ unless $v=1_\cV$, in which case each $g_{1,i} : 1_\cC\to 1_\cC$ is a scalar.
We conclude that 
$Z_\cC(\cC\boxdot_\cV \cD)(1\boxdot d_1\to 1\boxdot d_2)\cong \cD(d_1\to d_2)$.
\end{proof}

\begin{thm}
\label{thm:FactorizeViaCentralizers}
Suppose $\cC$ is a spherical fusion category, $\cE$ is any tensor category with simple unit, and $\cC\hookrightarrow \cE$ is a fully-faithful embedding.
The tensor product map $\otimes : Z_{\cC}(\cE)\boxtimes_{Z(\cC)} \cC\to \cE$ is fully faithful.
\end{thm}
\begin{proof}
We must show that whenever $c_1,c_2\in \cC$ and $e_1,e_2\in Z_\cC(\cE)$, we have
$$
\bigoplus_{z\in \Irr(Z(\cC))} 
Z_\cC(\cE)(e_1\to e_2\otimes z)
\otimes 
\cC(z\otimes c_1\to c_2)
\quad \cong\quad
\cE(e_1\otimes c_1\to e_2\otimes c_2).
$$
Using rigidity and semisimplicity of $\cC$, we may assume $c_1=1$ and $c_2=X:=\bigoplus_{c\in \Irr(\cC)}c$. 
Our question becomes:
$$
\bigoplus_{z\in \Irr(Z(\cC))} 
Z_\cC(\cE)(e_1\to e_2\otimes z)
\otimes 
\cC(z\to X)
\quad\stackrel{?}{\cong}\quad
\cE(e_1\to e_2\otimes X).
$$

The right hand side $\cE(e_1\to e_2\otimes X)$ carries a canonical action of Ocneanu's tube algebra $\Tube(\cC)$:
$$
\tikzmath{
\draw[thick, blue] (-.15,-.7) --node[left]{$\scriptstyle c$} (-.15,-.3);
\draw (.15,-.7) --node[right]{$\scriptstyle a$} (.15,-.3);
\draw[thick, blue] (.15,.7) --node[right]{$\scriptstyle c$} (.15,.3);
\draw (-.15,.7) --node[left]{$\scriptstyle b$} (-.15,.3);
\roundNbox{}{(0,0)}{.3}{.1}{.1}{$f$}
}
\rhd
\tikzmath{
\draw (0,-.7) --node[right]{$\scriptstyle e_1$} (0,-.3);
\draw (.15,.7) --node[right]{$\scriptstyle d$} (.15,.3);
\draw (-.15,.7) --node[left]{$\scriptstyle e_2$} (-.15,.3);
\roundNbox{}{(0,0)}{.3}{.1}{.1}{$\xi$}
}
=
\delta_{a=d}
\tikzmath{
\draw[thick, blue] (.15,1) to[out=-90,in=90] (-.9,.3) --node[left]{$\scriptstyle c$} (-.9,-.3) to[out=-90,in=-90] (.9,-.3) --node[right]{$\scriptstyle \overline{c}$} (.9,1.6) arc (0:180:.225);
\draw[knot] (0,-1.2) --node[right]{$\scriptstyle e_1$} (0,-.9) -- (0,-.3);
\draw (.45,1) --node[right]{$\scriptstyle d$} (.45,.3);
\draw (.15,1.6) --node[left]{$\scriptstyle b$} (.15,2);
\draw[knot] (-.45,2) --node[left]{$\scriptstyle e_2$} (-.45,.3);
\roundNbox{}{(0,0)}{.3}{.4}{.4}{$\xi$}
\roundNbox{}{(.3,1.3)}{.3}{.1}{.1}{$f$}
}
\qquad\qquad
f \in \cC(c\otimes a\to b\otimes c),\,\,
\xi \in \cE(e_1\to e_2\otimes d).
$$
We can decompose the above representation into irreps of $\Tube(\cC)$.
Recall from \cite{MR1782145,MR1966525} that $\Rep(\Tube(\cC)) \cong Z(\cC)^{\rm op}$, and every irrep is of the form $H_z = \cC(z\to X)$ where $z\in \Irr(Z(\cC))$, which carries a similar $\Tube(\cC)$-action as above.
Hence we can decompose
$$
\cE(e_1\to e_2\otimes X) 
\cong 
\bigoplus_{z\in \Irr(Z(\cC))} M_z \otimes H_z
=
\bigoplus_{z\in \Irr(Z(\cC))} M_z \otimes \cC(z\to X)
$$
where $M_z$ is a multiplicity space.
It remains to identify $M_z$ with $Z_\cC(\cE)(e_1\to e_2\otimes z)$.

Since $\cC\hookrightarrow \cE$ fully faithfully, $H_z = \cC(z\to X) \cong \cE(z\to X)$, so
$$
M_z
\cong
\Rep(\Tube(\cC))(
\cE(z\to X)
\to 
\cE(e_1\to e_2\otimes X)
).
$$
Observe that $\cE(z\to -)$ and $\cE(e_1\to e_2\otimes -)$ are both functors $\cE\to \Vec$.
By the Yoneda Lemma, $\Hom(\cE(z\to -) \to \cE(e_1\to e_2\otimes -)) \cong \cE(e_1\to e_2\otimes z)$ canonically.
Since maps of $\Tube(\cC)$ representations are maps between the underlying vector spaces which intertwine the $\Tube(\cC)$-actions, we see that $M_z$ is exactly the  subspace of $\cE(e_1\to e_2\otimes z)$
which intertwines the two $\Tube(\cC)$-actions, i.e.,  $Z_\cC(\cE)(e_1\to e_2\otimes z)$.
\end{proof}

Using the fact that $\cC'=Z_\cC(\Bim(R))$, we have the following immediate corollary.

\begin{cor}
\label{cor:FactorizeBimR}
$\Bim(R)\cong \cC' \boxtimes_{Z(\cC)} \cC$.
\end{cor}
\begin{proof}
After applying Theorem \ref{thm:FactorizeViaCentralizers}, the only thing that remains to show is that the functor $\cC' \boxtimes_{Z(\cC)}\cC\to \Bim(R)$ is dominant.
Indeed, for any object $X\in \Bim(R)$, $X$ is a summand of $\bigoplus_{c\in \Irr(\cC)} c\otimes X \otimes \overline c \in \cC'$ by \cite[Lem.~6.3]{MR4581741}.
\end{proof}

\subsection{The anchored planar algebra model for \texorpdfstring{$\cC\boxtimes_\cV\cD$}{C timesV D}}

In this section, we give a model $\cE$ for $\cC\boxtimes_\cV\cD$ using anchored planar algebras in the setting that $\Phi: \cV\to \cD$ admits a right adjoint (so that there is an anchored planar algebra associated to $\cD$).
We begin by defining a full subcategory $\cE_0$, and $\cE$ is the (unitary) Cauchy completion. 

\begin{itemize}
\item
Objects in $\cE_0$ are formal symbols 
$c \otimes x^{i}$
for $c\in \cC$ and $i\geq 0$. 
By convention, $x^0 = 1$.
\item
Morphism spaces are defined by
$$
\cE_0(
a \otimes x^{i}
\to
b \otimes x^{j}
)
:=
\cC(a \to b  \otimes \Psi \cP[j+i] ),
$$
where $\cP$ is the anchored planar algebra in $\cV$ corresponding to $(\cD,\Phi^Z:\cV\to Z(\cD),x)$.
We represent morphisms in the graphical calculus by
$$
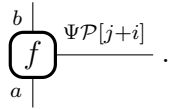
\tikzmath{
\roundNbox{fill=white}{(0,0)}{.3}{0}{0}{$f$}
\draw (0,-.7) --node[left]{$\scriptstyle a$} (0,-.3);
\draw (0,.7) --node[left]{$\scriptstyle b$} (0,.3);
\draw (.3,0) --node[above]{$\scriptstyle \Psi\cP[j+i]$} (1.6,0);
}\,.
$$
\item
Composition of morphisms is defined by
$$
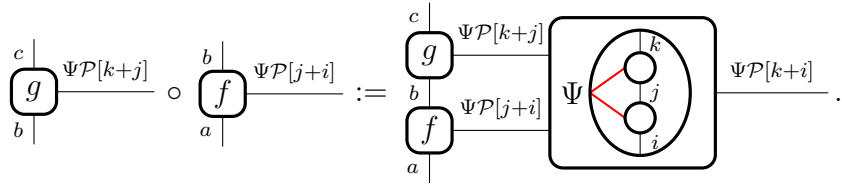
\tikzmath{
\roundNbox{fill=white}{(0,0)}{.3}{0}{0}{$g$}
\draw (0,-.7) --node[left]{$\scriptstyle b$} (0,-.3);
\draw (0,.7) --node[left]{$\scriptstyle c$} (0,.3);
\draw (.3,0) --node[above]{$\scriptstyle \Psi\cP[k+j]$} (1.6,0);
}
\circ
\tikzmath{
\roundNbox{fill=white}{(0,0)}{.3}{0}{0}{$f$}
\draw (0,-.7) --node[left]{$\scriptstyle a$} (0,-.3);
\draw (0,.7) --node[left]{$\scriptstyle b$} (0,.3);
\draw (.3,0) --node[above]{$\scriptstyle \Psi\cP[j+i]$} (1.6,0);
}
:=
\tikzmath{
	\coordinate (a) at (0,0);
	\pgfmathsetmacro{\boxWidth}{1};
	\roundNbox{unshaded}{(-2.8,.5)}{.3}{0}{0}{$g$}
	\roundNbox{unshaded}{(-2.8,-.5)}{.3}{0}{0}{$f$}
	\draw[rounded corners=5pt, very thick, unshaded] ($ (a) - (\boxWidth,\boxWidth) - (.2,0) $) rectangle ($ (a) + (\boxWidth,\boxWidth) $);
	\draw ($ (a) + 5/6*(0,1) $) -- ($ (a) - 5/6*(0,\boxWidth) $);
	\draw[thick, red] ($ (a) + 1/3*(0,\boxWidth) - 1/5*(\boxWidth,0) $) -- ($ (a) - 2/3*(\boxWidth,0) $);
	\draw[thick, red] ($ (a) - 1/3*(0,\boxWidth) - 1/5*(\boxWidth,0) $) -- ($ (a) - 2/3*(\boxWidth,0) $);
	\draw[very thick] (a) ellipse ({2/3*\boxWidth} and {5/6*\boxWidth});
	\filldraw[very thick, unshaded] ($ (a) + 1/3*(0,\boxWidth) $) circle (1/5*\boxWidth);
	\filldraw[very thick, unshaded] ($ (a) - 1/3*(0,\boxWidth) $) circle (1/5*\boxWidth);
	\node at ($ (a) + (.2,0) $) {\scriptsize{$j$}};
	\node at ($ (a) + (.2,-.65) $) {\scriptsize{$i$}};
	\node at ($ (a) + (.2,.65) $) {\scriptsize{$k$}};
  \node at ($ (a) - (\boxWidth,0) + (.1,0) $) {$\Psi$};
\draw (-2.8,-1.2) --node[left]{$\scriptstyle a$} (-2.8,-.8);
\draw (-2.8,-.2) --node[left]{$\scriptstyle b$} (-2.8,.2);
\draw (-2.8,1.2) --node[left]{$\scriptstyle c$} (-2.8,.8);
\draw (-1.2,.5) --node[above]{$\scriptstyle \Psi\cP[k+j]$} (-2.5,.5);
\draw (-1.2,-.5) --node[above]{$\scriptstyle \Psi\cP[j+i]$} (-2.5,-.5);
\draw (1,0) --node[above]{$\scriptstyle \Psi\cP[k+i]$} (2.5,0);
}\,.
$$
\item
Tensor product is given on objects by $(a\otimes x^i)\otimes (b\otimes x^j):= a\otimes b \otimes x^{i+j}$, and
on morphisms by
$$
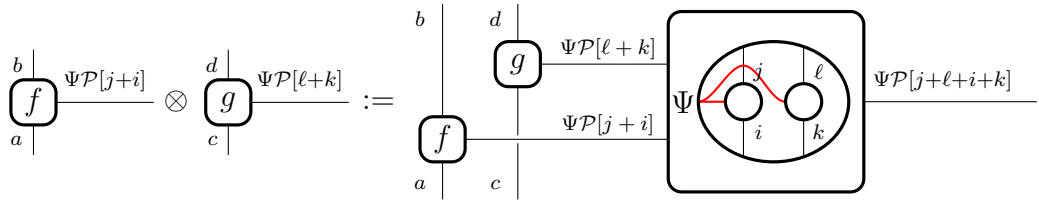
\tikzmath{
\roundNbox{fill=white}{(0,0)}{.3}{0}{0}{$f$}
\draw (0,-.7) --node[left]{$\scriptstyle a$} (0,-.3);
\draw (0,.7) --node[left]{$\scriptstyle b$} (0,.3);
\draw (.3,0) --node[above]{$\scriptstyle \Psi\cP[j+i]$} (1.6,0);
}
\otimes
\tikzmath{
\roundNbox{fill=white}{(0,0)}{.3}{0}{0}{$g$}
\draw (0,-.7) --node[left]{$\scriptstyle c$} (0,-.3);
\draw (0,.7) --node[left]{$\scriptstyle d$} (0,.3);
\draw (.3,0) --node[above]{$\scriptstyle \Psi\cP[\ell+k]$} (1.6,0);
}
:=
\begin{tikzpicture}[baseline=-.6cm]
	\draw (0,0) -- (2,0);
	\draw (4.4,-.5) --node[above]{$\scriptstyle \Psi\cP[j+\ell+i+k]$} (6.9,-.5);
	\draw (0,.8) -- (0,-1.8);
	\draw (-1,.8) -- (-1,-1.8);
	\draw[knot] (-1,-1) -- (2,-1);
	\roundNbox{unshaded}{(-1,-1)}{.3}{0}{0}{$f$}
	\roundNbox{unshaded}{(0,0)}{.3}{0}{0}{$g$}
	\PsiTensor{(3.4,-.5)}{1.2}{i}{j}{k}{\ell}
	\node at (1.2,.2) {\scriptsize{$\Psi\cP[\ell+k]$}};
	\node at (1.2,-.8) {\scriptsize{$\Psi\cP[j+i]$}};
	\node at (-1.3,-1.6) {\scriptsize{$a$}};
	\node at (-.3,-1.6) {\scriptsize{$c$}};
	\node at (-1.3,.6) {\scriptsize{$b$}};
	\node at (-.3,.6) {\scriptsize{$d$}};
\end{tikzpicture}
$$
where the crossing is the half-braiding of $\Psi^Z\cP[j+i]\in Z(\cC)$ with $c\in \cC$ (this choice ensures that the box labelled $g$ can pass under the crossing).
Since $\Psi$ is reverse-braided, we have
\begin{align*}
\tikzmath{
	\draw (0,-1) to[out=0,in=180] (1,0) --node[above, xshift=-.2cm]{$\scriptstyle \Psi\cP[\ell+k]$} (2,0);
	\draw[knot] (0,0) to[out=0,in=180] (1,-1) --node[below, xshift=-.2cm]{$\scriptstyle \Psi\cP[j+i]$} (2,-1);
	\draw (4.6,-.5) --node[above]{$\scriptstyle \Psi\cP[j+\ell+i+k]$} (6.6,-.5);
	\PsiTensor{(3.4,-.5)}{1.2}{i}{j}{k}{\ell}
}
&=
\Psi\left(
\tikzmath{
	\draw (0,0) to[out=0,in=180] (1,-1) --node[below, xshift=-.2cm]{$\scriptstyle \Psi\cP[j+i]$} (2,-1);
	\draw[knot] (0,-1) to[out=0,in=180] (1,0) --node[above, xshift=-.2cm]{$\scriptstyle \Psi\cP[\ell+k]$} (2,0);
	\draw (4.4,-.5) --node[above]{$\scriptstyle \cP[j+\ell+i+k]$} (6.2,-.5);
	\tensor{(3.2,-.5)}{1.2}{i}{j}{k}{\ell}
}
\right)
\\&=
\Psi\left(
\tikzmath{
	\draw (.5,-1) --node[above]{$\scriptstyle \cP[\ell+k]$} (2,-1);
	\draw (.5,0) --node[above]{$\scriptstyle \cP[j+i]$} (2,0);
	\draw (4.4,-.5) --node[above]{$\scriptstyle \cP[j+\ell+i+k]$} (6.2,-.5);
	\RevTensor{(3.2,-.5)}{1.2}{i}{j}{k}{\ell}
}
\right)
\\&=
\tikzmath{
	\draw (.5,-1) --node[above]{$\scriptstyle \Psi\cP[\ell+k]$} (2,-1);
	\draw (.5,0) --node[above]{$\scriptstyle \Psi\cP[j+i]$} (2,0);
	\draw (4.6,-.5) --node[above]{$\scriptstyle \Psi\cP[j+\ell+i+k]$} (6.6,-.5);
	\PsiRevTensor{(3.4,-.5)}{1.2}{i}{j}{k}{\ell}
}
\end{align*}
ensuring that the exchange axiom $(g_1\circ f_1)\otimes(g_2\circ f_2)=(g_1\otimes g_2)\circ(f_1 \otimes f_2)$ is satisfied (see \cite[\S 6.3]{MR4528312} for a similar computation).

The associators are induced from those of $\cC$:
$$
\alpha_{
a \otimes x^{i}
,
b \otimes x^{j}
,
c \otimes x^{k}
}
:=
\tikzmath{
\draw (0,-.7) --node[left]{$\scriptstyle (a\otimes b)\otimes c$} (0,-.3);
\draw (0,.3) --node[left]{$\scriptstyle a\otimes(b\otimes c)$} (0,.7);
\draw (1.8,0) --node[above]{$\scriptstyle \Psi\cP[2r]$} (2.9,0);
\roundNbox{unshaded}{(0,0)}{.3}{0}{0}{$\alpha$}
	\identityMap{(1.4,0)}{.4}{\ell\,\,\,}
}
\qquad\qquad
\ell:=i+j+k,
$$
and similarly for the unitors.
\item
$\cE_0$ is rigid with duals given by $(c \otimes x^{i})^\vee:=c^\vee \otimes x^{i}$,
and evaluation and coevaluation are given by
$$
\ev_{c \otimes x^i}
=
\begin{tikzpicture}[baseline=-.1cm]
	\draw (-.6,-.8) node[above, yshift=-12, xshift=1]{$\scriptstyle c^\vee$} -- (-.6,-.4) arc (180:0:.3cm) -- (0,-.8)node[above, yshift=-12]{$\scriptstyle c$};
	\draw (1.7,0) --node[above]{$\scriptstyle\Psi\cP[2i]$} (2.7,0);
	\node[xshift=1] at (-.3,.1) {\scriptsize{$\ev_{c}$}};
	\evaluationMap{(1.2,0)}{.5}{n}
\end{tikzpicture}
\qquad\qquad
\coev_{c \otimes x^i}
=
\begin{tikzpicture}[baseline=-.1cm]
	\draw (-.6,.8)node[above]{$\scriptstyle c$} -- (-.6,.4) arc (-180:0:.3cm) -- (0,.8)node[above, xshift=2]{$\scriptstyle c^\vee$};
	\draw (1.7,0) --node[above]{$\scriptstyle\Psi\cP[2i]$} (2.7,0);
	\node at (-.3,-.1) {\scriptsize{$\coev_{c}$}};
	\coevaluationMap{(1.2,0)}{.5}{n}
\end{tikzpicture}\,.
$$
In the unitary setting, the evaluation and coevaluation in $\cC$ come from our chosen unitary dual functor on $\cC$.

\item
The pivotal structure $\varphi^\cE:c \otimes x^{\otimes i} \to (c \otimes x^{i})^{\vee\vee}$ is given by
$$
\tikzmath{
	\draw (0,-1) -- (0,1);
	\roundNbox{unshaded}{(0,0)}{.4}{0}{0}{$\varphi_{c}$}
	\draw (1.8,0) --node[above]{$\scriptstyle \Psi\cP[2i]$} (2.8,0);
	\node at (-.3,-.8) {\scriptsize{$c$}};
	\node at (-.4,.8) {\scriptsize{$c^{\vee\vee}$}};
	\identityMap{(1.4,0)}{.4}{i\,\,\,}
}\,.
$$
\end{itemize}

\begin{rem}
\label{rem:EFusion}
When $\Phi: \cV\to \cD$ is fully faithful and $\cC$ has simple unit object, then the unit of $\cE_0$ is also simple. Indeed, in that case, $\cP[0]=1_\cV$, and hence
$$
\End_{\cE_0}(1_\cC\otimes x^0) = \cC(1_\cC\to 1_\cC \otimes \Psi\cP[0]) \cong \cC(1_\cC\to 1_\cC)\cong \bbC.
$$
\end{rem}

\begin{rem}
\label{rem:Comodules}
When $\Psi: \cV\to \cC$ is dominant and admits a left adjoint $I$, as in \cite{mitchell-thesis}, $\cE$ can be identified with
$\mathsf{coMod}_\cD(\Phi A)$ for $A = I(1_\cC)$. 
Indeed,
\begin{align*}
\cE_0\big(\Psi u \otimes x^i \to \Psi v \otimes x^j\big)
&=
\cC\big(\Psi u \to \Psi v \otimes \Psi\cP[j+i]\big)
\\&\cong
\cC\big(1_\cC \to \Psi (\overline{u}\otimes v \otimes \cP[j+i])\big)
\\&\cong
\cV\big(A \to \overline{u}\otimes v \otimes \Tr_\cV(x^{j+i})\big)
\\&\cong
\cV\big(A \to \Tr_\cV(x^i\otimes \Phi\overline{u} \otimes \Phi v \otimes x^{j})\big)
\\&=
\cV\big(A \to \underline{\Hom}_\cD(\Phi u \otimes x^i\to \Phi v \otimes x^{j})\big)
\\&=
\cD\big(\Phi u \otimes x^i\otimes \Phi A \to \Phi v \otimes x^{j}\big)
\\&\cong
\mathsf{coMod}_\cD(A)\big(\Phi u \otimes x^i\otimes \Phi A \to \Phi v \otimes x^{j}\otimes  \Phi A\big)
\end{align*}
where the $x$ which appears in the left hand side is a formal symbol, and the $x$ which appears in the right hand side is our chosen generator of $\cD$.
\end{rem}

\begin{prop}
\label{prop:BalancedTensorProduct}
The monoidal category $\cE$ constructed above is canonically equivalent, as monoidal category, to the balanced tensor product $\cC\boxtimes_\cV \cD$.
\end{prop}

\begin{proof}
The category $\cE$ is the idempotent completion of $\cE_0$, and
the balanced tensor product $\cC\boxtimes_\cV \cD$ is canonically equivalent to the idempotent completion of the ladder category $\cC\boxdot_\cV \cD_0$,
where $\cD_0\subset\cD $ is the full subcategory on the objects of the form $\Phi v\otimes x^i$.
So it is enough to construct an equivalence $F_0$ of monoidal categories from $\cE_0$ to $\cC\boxdot_\cV \cD_0$.
We define $F_0$ by 
$$
F_0(c\otimes x^i):=c\boxdot x^i.
$$
On morphism spaces, we define $F_0$ by the following sequence of isomorphisms:
\begin{align*}
\cE_0(a\otimes x^i \to b\otimes x^j)
&=
\cC(a\to b\otimes \Psi\cP[j+i])
\displaybreak[1]\\&\cong
\bigoplus_{v\in \Irr(\cV)}\cC(a\to b\otimes \Psi v)\otimes \cV(v\to \cP[j+i])
\displaybreak[1]\\&=
\bigoplus_{v\in \Irr(\cV)}\cC(a\to b\otimes \Psi v)\otimes \cV(v\to\Tr_\cV(x^{j+i}))
\displaybreak[1]\\&\cong
\bigoplus_{v\in \Irr(\cV)}\cC(a\to b\otimes \Psi v)\otimes \cD(\Phi v\to x^{j+i})
\displaybreak[1]\\&\cong
\bigoplus_{v\in \Irr(\cV)}
\cC(a\to b\otimes \Psi v)\otimes \cD(\Phi v \otimes x^i\to x^j)
\displaybreak[1]\\&=
(\cC\boxdot_\cV \cD_0)\big(a\boxdot x^i \to b\boxdot x^j\big).
\end{align*}
The first isomorphism in the second line above uses the fact that each object $w\in \cV$ can be canonically decomposed into simple objects as $w=\bigoplus_{v\in \Irr(\cV)}\cV(v\to w)\otimes v$.
It is a somewhat tricky exercise to show that $F_0$ is a tensor functor, which is then automatically fully faithful.
Essential surjectivity follows since $a\boxdot (\Phi v \otimes x^i)\cong (a\otimes \Psi  v)\boxdot x^i$.
\end{proof}

\subsection{The unitary setting}

We now consider the case that $\cV$ is a unitary tensor category equipped with a unitary dual functor $\cV$.
Suppose we have a pointed unitary $\cV$-module multitensor category $(\cD,\Phi^Z:\cV\to Z(\cD),x,\vee_\cD,\psi_\cD)$ and a unitary $\cV^{\rev}$-module multitensor category $(\cC,\Psi^Z: \cV\to Z(\cC),\vee_\cC,\psi_\cC)$ where $\Psi^Z$ is a reverse-braided functor.
We further assume that $\Phi: \cV\to \cD$ 
admits a unitary adjoint \cite[\S2.1]{2301.11114}.

As in \cite[\S5.2.1]{2301.11114}, we can promote $\cE$ to a unitary multitensor category equipped with a unitary dual functor and canonical state.
This amounts to the following tasks:
\begin{enumerate}[label=(E\arabic*)]
\item 
\label{E:Dagger}
construct a $\dag$-structure on $\cE_0$ under which is is a unitary multitensor category,
\item
\label{E:State}
construct a faithful state $\psi_\cE$ on $\End_\cE(1_\cE)$, and
\item
\label{E:UDF}
check $\vee_\cE$ is unitary, and the canonical unitary pivotal structure induced by $\vee_\cE$ is $\varphi^\cE$.
\end{enumerate}

To accomplish \ref{E:Dagger} above, we define $\dag$ on $\cE_0$ as in \cite[(25)]{2301.11114}, but we apply $\Psi$ when necessary to push objects and morphisms from $\cV$ into $\cC$:
$$
\left(
\tikzmath{
\roundNbox{fill=white}{(0,0)}{.3}{0}{0}{$f$}
\draw (0,-.7) --node[left]{$\scriptstyle a$} (0,-.3);
\draw (0,.7) --node[left]{$\scriptstyle b$} (0,.3);
\draw (.3,0) --node[above]{$\scriptstyle \Psi\cP[n+k]$} (1.5,0);
}
\right)^*
=
\tikzmath{
\roundNbox{fill=white}{(0,0)}{.3}{0}{0}{$f^*$}
\draw (0,-.7) --node[left]{$\scriptstyle b$} (0,-.3);
\draw (0,.7) --node[left]{$\scriptstyle a$} (0,.3);
\draw (.3,0) --node[above]{$\scriptstyle \Psi\cP[n+k]$} (1.5,0);
}
:=
\tikzmath{
\roundNbox{fill=white}{(0,0)}{.3}{.1}{.1}{$f^\dag$}
\roundNbox{fill=white}{(1,0)}{.3}{.2}{.2}{$\scriptstyle \Psi r_{n+k}^{-1}$}
\draw (-.2,-.7) --node[left]{$\scriptstyle b$} (-.2,-.3);
\draw (0,.7) --node[left]{$\scriptstyle a$} (0,.3);
\draw (.2,-.3) 
arc (-180:0:.4cm) node[right, yshift=-.2cm]{$\scriptstyle \overline{\Psi\cP[n+k]}$};
\draw (1,.3) arc (180:90:.3cm) -- (1.6,.6) node[right]{$\scriptstyle \Psi\cP[k+n]$};
}\,.
$$
The proof that this defines a dagger structure on $\cE_0$ is entirely similar to the one in \cite{2301.11114}.

Since $\Psi$ is a unitary pivotal tensor functor between unitary multitensor categories equipped with unitary dual functors,
it automatically preserves the unitary pivotal structure, and thus $\Psi \coev_{\cP[i+j]}^\dag = \coev_{\Psi \cP[i+j]}^\dag$.
Thus by \cite[Lem.~5.9]{2301.11114}, the sesquilinear form on $\cE_0(a \otimes x^i \to b\otimes x^j)$
$$
\langle f, g\rangle
:=
\tikzmath{
	\coordinate (a) at (.2,0);
	\pgfmathsetmacro{\boxWidth}{1};
\draw (-2.8,-.2) --node[left]{$\scriptstyle a$} (-2.8,.2);
\draw (-1,.5) --node[above]{$\scriptstyle \Psi \cP[i{+}j]$} (-2.2,.5) -- (-2.5,.5);
\draw (-1,-.5) --node[above]{$\scriptstyle \Psi \cP[j{+}i]$} (-2.2,-.5) -- (-2.5,-.5);
\draw[knot] (-2.8,.8) node[left, yshift=.2cm, xshift=.1cm]{$\scriptstyle b$} arc (180:0:.3cm) -- (-2.2,-.8) node[right]{$\scriptstyle b^\vee$} arc (0:-180:.3cm) node[left, yshift=-.2cm, xshift=.1cm]{$\scriptstyle b$};
	\roundNbox{unshaded}{(-2.8,.5)}{.3}{0}{0}{$f$}
	\roundNbox{unshaded}{(-2.8,-.5)}{.3}{0}{0}{$g^*$}
	\draw[rounded corners=5pt, very thick, unshaded] ($ (a) - (\boxWidth,\boxWidth) - (.2,0) $) rectangle ($ (a) + (\boxWidth,\boxWidth) $);
	\draw ($ (a) + 1/3*(0,1) $) -- ($ (a) - 1/3*(0,\boxWidth) $);
	\draw[thick, red] ($ (a) + 1/3*(0,\boxWidth) - 1/5*(\boxWidth,0) $) -- ($ (a) - 2/3*(\boxWidth,0) $);
	\draw[thick, red] ($ (a) - 1/3*(0,\boxWidth) - 1/5*(\boxWidth,0) $) -- ($ (a) - 2/3*(\boxWidth,0) $);
	\draw[very thick] (a) ellipse ({2/3*\boxWidth} and {5/6*\boxWidth});
	\filldraw[very thick, unshaded] ($ (a) + 1/3*(0,\boxWidth) $) circle (1/5*\boxWidth);
	\filldraw[very thick, unshaded] ($ (a) - 1/3*(0,\boxWidth) $) circle (1/5*\boxWidth);
	\node at ($ (a) + (.3,0) $) {$\scriptstyle j{+}i$};
	\node at ($ (a) - (\boxWidth,0) $) {$\Psi$};
\draw[densely dotted] (1.2,0) --node[above]{$\scriptstyle 1_\cD$}  (2,0);
}
=
\tikzmath{
\roundNbox{unshaded}{(0,.5)}{.3}{.1}{.1}{$f$}
\roundNbox{unshaded}{(0,-.5)}{.3}{.1}{.1}{$g^\dag$}
\draw (-.2,.8) node[left, yshift=.2cm]{$\scriptstyle b$} 
.. controls ++(90:.7cm) and ++(90:.7cm) .. (2,.8) --node[right]{$\scriptstyle b^\vee$}(2,-.8) 
.. controls ++(270:.7cm) and ++(270:.7cm) ..
(-.2,-.8) node[left, yshift=-.2cm]{$\scriptstyle b$};
\draw (.2,.8) .. controls ++(90:.5cm) and ++(90:.5cm) .. 
(1.6,.8) --node[left,xshift=.1cm]{$\scriptstyle \Psi \cP[n]^\vee$} (1.6,-.8)
.. controls ++(270:.5cm) and ++(270:.5cm) .. (.2,-.8);
\draw (0,-.2) --node[left]{$\scriptstyle a$} (0,.2);
}
$$
is positive definite.
This immediately implies as in \cite[Prop.~5.10]{2301.11114} that the $2\times 2$ linking algebras
$$
L:=
\begin{pmatrix}
\cE_0(a\otimes x^{i} \to a\otimes x^{i})
&
\cE_0(b\otimes x^{j} \to a\otimes x^{i})
\\
\cE_0(a\otimes x^{i} \to b\otimes x^{j})
&
\cE_0(b\otimes x^{j} \to b\otimes x^{j})
\end{pmatrix}
$$
are finite dimensional $\rm C^*$-algebras, proving $\cE_0$ is $\rm C^*$.
We have thus established \ref{E:Dagger}.

The construction of a faithful state $\psi_\cE$ on $\End_\cE(1_\cE)$ is similar to \cite[Cor.~5.12]{2301.11114}
$$
\psi_\cE\left(
\tikzmath{
\roundNbox{fill=white}{(0,0)}{.3}{0}{0}{$f$}
\draw[densely dotted] (0,-.7) --node[left]{$\scriptstyle 1_\cC$} (0,-.3);
\draw[densely dotted] (0,.7) --node[left]{$\scriptstyle 1_\cC$} (0,.3);
\draw (.3,0) --node[above]{$\scriptstyle \Psi\cP[0]$} (1.3,0);
}
\right)
:=
\tikzmath{
\roundNbox{fill=white}{(0,0)}{.3}{0}{0}{$f$}
\roundNbox{fill=white}{(1.7,0)}{.3}{.2}{.2}{$\Psi\psi_\cP$}
\draw[densely dotted] (0,-.7) --node[left]{$\scriptstyle 1_\cC$} (0,-.3);
\draw[densely dotted] (0,.7) --node[left]{$\scriptstyle 1_\cC$} (0,.3);
\draw (.3,0) --node[above]{$\scriptstyle \Psi\cP[0]$} (1.2,0);
\draw[densely dotted] (2.2,0) --node[above]{$\scriptstyle 1_\cC$} (2.6,0);
}\,,
$$
establishing \ref{E:State}.
Finally, that $\vee_\cE$ is unitary and induces $\varphi^\cE$ is similar to \cite[Prop.~5.13]{2301.11114} adding $\Psi$s when necessary, establishing \ref{E:UDF}.

By the same argument as \cite[Prop.~5.15]{2301.11114},
when $\vee_\cV$ and $\psi_\cP$ are spherical, then so is $\psi_\cE$.
In fact, instead of demanding that $\psi_\cP$ be spherical, the following weaker condition is sufficient (in addition to sphericality of $\vee_\cV$):
$$
\tikzmath{
	\coordinate (a) at (0,0);
	\pgfmathsetmacro{\boxWidth}{1};
	\draw (-2.4,0) --node[above]{$\scriptstyle \cP[2]$} (-1.4,0);
	\draw (1,0) --node[above]{$\scriptstyle \cP[0]$} (1.7,0);
	\draw[dotted] (2.6,0) --node[above]{$\scriptstyle 1_\cC$} (3.3,0);
	\roundNbox{unshaded}{(2.3,0)}{.3}{.3}{0}{$\Psi\psi_\cP$}
	\draw[rounded corners=5pt, very thick, unshaded] ($ (a) - (\boxWidth,\boxWidth) - (.4,0) $) rectangle ($ (a) + (\boxWidth,\boxWidth) $);
  \node at (-1.1,0) {$\Psi$};
	\draw[thick, red] (180:1/5*\boxWidth) -- (180:4/5*\boxWidth);
	\draw[very thick] (a) circle (4/5*\boxWidth);
	\draw[very thick] (a) circle (1/5*\boxWidth);
	\draw (90:1/5*\boxWidth) .. controls ++(90:.3cm) and ++(90:.5cm) .. (180:1/2*\boxWidth) .. controls ++(270:.5cm) and ++(270:.3cm) ..  (270:1/5*\boxWidth);
}
=
\tikzmath{
	\coordinate (a) at (0,0);
	\pgfmathsetmacro{\boxWidth}{1};
	\draw (-2.4,0) --node[above]{$\scriptstyle \cP[2]$} (-1.4,0);
	\draw (1,0) --node[above]{$\scriptstyle \cP[0]$} (1.7,0);
	\draw[dotted] (2.6,0) --node[above]{$\scriptstyle 1_\cC$} (3.3,0);
	\roundNbox{unshaded}{(2.3,0)}{.3}{.3}{0}{$\Psi\psi_\cP$}
	\draw[rounded corners=5pt, very thick, unshaded] ($ (a) - (\boxWidth,\boxWidth) - (.4,0) $) rectangle ($ (a) + (\boxWidth,\boxWidth) $);
  \node at (-1.1,0) {$\Psi$};
	\draw[thick, red] (180:1/5*\boxWidth) -- (180:4/5*\boxWidth);
	\draw[very thick] (a) circle (4/5*\boxWidth);
	\draw[very thick] (a) circle (1/5*\boxWidth);
	\draw (90:1/5*\boxWidth) .. controls ++(90:.3cm) and ++(90:.5cm) .. (0:1/2*\boxWidth) .. controls ++(270:.5cm) and ++(270:.3cm) ..  (270:1/5*\boxWidth);
}\,.
$$

\section{Classification of finite depth objects in \texorpdfstring{$\cC'$}{C'}
}

Let $\cC$ be a unitary fusion category fully faithfully embedded in $\Bim(R)$ where $R$ is a hyperfinite factor of type $\mathrm{II}_1$, $\mathrm{II}_\infty$, or $\mathrm{III}_1$.
We denote by $\cC'$ its commutant category defined in \S\ref{sec:BicommutantCats} above.

Recall that our main goal is to prove Conjecture \ref{conjIntro} in the case that $\cB = \cC^\prime$: finite depth real (symmetrically self-dual) objects of $\cC'$ up to conjugation by invertible elements of $\cC'$ are in bijective correspondence with connected finite depth unitary anchored planar algebras in $\cZ(\cC)^{\rev}$ up to isomorphism.

\subsection{From finite depth objects to anchored planar algebras}

We first fix a symmetrically self-dual object $m \in \cC'$, and we construct a connected unitary anchored planar algebra $\cP$ in $\cZ(\cC)^{\rev} \cong \cZ(\cC')$.
Since $m\in \cC'$, it is equipped with a half-braiding, which we'll call $\{\eta_{m,d}:m\otimes d\to d\otimes m\}_{d\in \cC}$.

Let $\cM:=\langle m, Z(\cC)\rangle$ be the tensor $\rm C^*$-subcategory of $\cC'$ generated by $m$ and the image of $Z(\cC)$ in $\cC'$ (under the operations of tensor product, orthogonal direct sums, and orthogonal direct summands).
By Corollary \ref{cor:M is a Z(C)rev module tensor cat}, there is a canonical $Z(\cC)^{\rev}$-module tensor category structure on $\cM$.
We equip $\cM$ with its unique spherical unitary dual functor.
Finally, we obtain a unitary anchored planar algebra $\cP$ from the unitary pivotal $Z(\cC)$-module tensor category $(\cM,m)$ by \cite{2301.11114}.

Observe that when $m$ has finite depth, then $\cM$ is fusion by Lemma \ref{lem:CommutingFusionCategoriesGenerateFusion} below.
In this case, $\cP$ is finite depth.

\begin{lem}
\label{lem:CommutingFusionCategoriesGenerateFusion}
Let $\cM$ be a semisimple rigid tensor category with simple unit object, and suppose $\cC,\cD\subset \cM$ are fusion subcategories that generate $\cM$ under tensor products and subobjects.
If $c\otimes d \cong d\otimes c$ for every $c\in \cC$ and $d\in \cD$, then $\cM$ is again fusion.
\end{lem}
\begin{proof}
A complete set of simples for $\cM$ is obtained by taking all distinct simple summands of the finitely many objects of the form $c\otimes d$ where $c\in \Irr(\cC)$ and $d\in \Irr(\cD)$.
\end{proof}

\subsection{From anchored planar algebras to finite depth objects}

Suppose $\cP$ is a connected unitary anchored planar algebra in $Z(\cC)^{\rev}$.
By the main result of our previous paper \cite{2301.11114}, $\cP$ corresponds to a pointed unitary module tensor category $(\cM,m)$.
By the construction from \S\ref{sec:BalancedTensor}, noting that the forgetful functor $Z(\cC)^{\rev}\to \cC$ is reverse braided, we can then form the balanced tensor product $\cE:=\cC\boxtimes_{Z(\cC)^{\rev}} \cM$.
It will be convenient to also think of $\cE=\cM \boxtimes_{Z(\cC)} \cC$ during this construction.
By Construction \ref{construction:CanonicalCentralizing}, $\cE$ can be equipped with a canonical centralizing structure $\{\sigma_{m,c}: m\otimes c \to c\otimes m\}_{m\in \cM, c\in \cC}$.

Since $\cC$ is fusion, it has simple unit object.
Since $\cP$ is connected, $\Phi:Z(\cC)\to \cM$ is fully faithful.
Thus 
$\cE$ has simple unit object
and
$\cC\hookrightarrow \cE=\cM\boxtimes_{Z(\cC)}\cC$ is fully faithful by Lemma \ref{lem:CanonicalInclusionsFF}.
When $\cP$ has finite depth, $\cM$ is fusion, and thus $\cE$ is also fusion by 
Corollary \ref{cor:RelativeProductFusion}.
$$
\begin{tikzcd}
&
Z(\cC)^{\rev} 
\arrow[r,"\Forget"]
\arrow[d,hookrightarrow,"\Phi"]
&
\cC
\arrow[d,hookrightarrow]
\arrow[dr,hookrightarrow]
\\
\{m\}
\arrow[r] 
& 
\cM
\arrow[r] 
&
\cE
\arrow[r,hookrightarrow,"\alpha"]
&
\Bim(R)
\end{tikzcd}
$$

As reviewed in \cite[\S 3.2]{MR4581741}, by \cite{MR1055708,MR1339767,MR3635673,MR4236062}, there is a fully faithful unitary tensor functor $\alpha:\cE\to \Bim(R)$, which is unique up to conjugation by an invertible object of $\Bim(R)$.
The restriction of $\alpha|_\cC:\cC\to \Bim(R)$ is similarly unique up to conjugation by an invertible object of $\Bim(R)$.
Since $\cC$ was alreday given to us as a full subcategory of $\Bim(R)$, conjugating by a suitable invertible object of $\Bim(R)$, we may assume that $\alpha|_\cC$ agrees with our initial presentation of $\cC$ (that is, $\alpha|_\cC=\id_\cC$).
Such tensor functors $\alpha:\cE\to \Bim(R)$ satisfying $\alpha|_\cC=\id_\cC$ are unique up to conjugation by an invertible object of $\cC'$.

Since $\cC\hookrightarrow \cE$ is fully faithful,
our centralizing structure $\sigma$ canonically promotes the image of each object $n\in \cM$ in $\cE$ to the relative center $Z_\cC(\cE)$, i.e., objects in $\cE$ which are equipped with half-braidings with objects in $\cC$.
Taking the image of $\sigma$ in $\Bim(R)$, we immediately get a unitary tensor functor $\cM \to \cC'$ such that the following diagram commutes:
$$
\begin{tikzcd}
&
Z(\cC)^{\rev} 
\arrow[r,"\Forget"]
\arrow[d,hookrightarrow,"\Phi"]
&
\cC
\arrow[d,hookrightarrow]
\arrow[dr,hookrightarrow]
\\
\{m\}
\arrow[r] 
& 
\cM
\arrow[dr, hookrightarrow]
\arrow[r] 
&
\cE
\arrow[r,hookrightarrow,"\alpha"]
&
\Bim(R)
\\
&&
Z_\cC(\cE)
\arrow[u,swap,"\Forget"]
\arrow[r, hookrightarrow]
&
\cC'.
\arrow[u,swap,"\Forget"]
\end{tikzcd}
$$
Now the image of $m$ in $\cC'$ is our desired symmetrically self-dual finite depth object.

The functor $\cM\hookrightarrow Z_\cC(\cE)$ is fully faithful by Lemma~\ref{lem:RelativeCenterFF}. It follows that $\cM\to \cC'$ is also fully faithful.

\subsection{Uniqueness of \texorpdfstring{$\cM\to \cC'$}{M->C'}}

Let $\cM$ be a unitary $Z(\cC)^{\rev}$-module tensor category such that the composite unitary tensor functor $\Phi : Z(\cC)\to \cM$ is fully faithful (in terms of the associated anchored planar algebra, this is the condition that $\cP[0]$ is connected).
Our next goal is to show that given two fully faithful $Z(\cC)^{\rev}$-central unitary tensor functors $\cM\to \cC'$ (a.k.a.~functors of $Z(\cC)^{\rev}$-module tensor categories), there exists an invertible object of $\cC'$ that conjugates one into the other.

Consider $\cC\subset \Bim(R)$.
We claim that the unitary tensor functor $\cM \hookrightarrow \cC' \to \Bim(R)$ and the inclusion $\cC\subset \Bim(R)$ assemble to a fully faithful unitary tensor functor $\beta:\cE=\cM \boxtimes_{Z(\cC)}\cC \to \Bim(R)$ such that $\beta|_\cC=\id_\cC$.

Observe that the underlying $R-R$ bimodule of an object of $Z(\cC)^{\rev}$ lives in $\cC\subset \Bim(R)$.
Moreover, our tensor functor $G:\cM \to \cC'$ is a morphism of unitary $Z(\cC)^{\rev}$-module tensor categories.
This means that we have a unitary action-coherence morphism $\gamma : \Phi_2\Rightarrow G\Phi_1$
where $\Phi_1: Z(\cC)^{\rev}\to \cM$ and $\Phi_2: Z(\cC)^{\rev}\to \cC'$ satisfying the coherence conditions from \cite[Def.~3.2]{MR4528312} (see also \cite[Def.~3.3]{2301.11114}).
Since $\Forget\circ\Phi_2$ is identically the forgetful functor $Z(\cC)\to \cC$,
the external square of the following digram 
$$
\begin{tikzcd}
&
Z(\cC)^{\rev} 
\arrow[r,"\Forget"]
\arrow[d,hookrightarrow,"\overset{\gamma}{\Longrightarrow}\,\,\,",swap]
\arrow[dl,hookrightarrow,swap,"\Phi_1"]
&
\cC
\arrow[d,hookrightarrow,"\Phi_2"]
\\
\cM
\arrow[r,hookrightarrow, "G"]
&
\cC'
\arrow[r,"\Forget"]
&
\Bim(R)
\end{tikzcd}
$$
commutes on the nose.
Thus the image of $Z(\cC)^{\rev}$ inside the images of $\cC$ and $\cM$ in $\Bim(R)$ are identical.
This equips the canonical map $\cC\boxtimes \cM \to \Bim(R)$ (which is the identity on $\cC$) with a $Z(\cC)^{\rev}$-balancing structure.

This map thus descends to a map $\beta:\cC\boxtimes_{Z(\cC)^{\rev}}\cM \to \Bim(R)$ (which is still the identity on $\cC$) by the universal property of the relative tensor product.
As in the previous section, we identify $\cC\boxtimes_{Z(\cC)^{\rev}}\cM = \cM \boxtimes_{Z(\cC)}\cC=\cE$.
Since the inclusion $\cM \to \cC'$ is fully faithful, $\cE$ is a full subcategory of $\cC'\boxtimes_{Z(\cC)}\cC$, the latter of which is equivalent to $\Bim(R)$ by Corollary \ref{cor:FactorizeBimR}.
We conclude the map $\beta:\cE\to \Bim(R)$ is fully faithful.

We now have two fully faithful unitary tensor functors $\alpha, \beta: \cE\to \Bim(R)$.
As reviewed in \cite[\S 3.2]{MR4581741}, by \cite{MR1055708,MR1339767,MR3635673,MR4236062}, there is an isomorphism of representations $(\Phi,\phi):\alpha \to \beta$, i.e., an invertible $\Phi\in \Bim(R)$ equipped with a natural family of unitary isomorphisms
$$
\big\{\phi_e : \Phi \boxtimes \alpha(e) \to \beta(e)\boxtimes \Phi\big\}_{e\in\cE}
$$
satisfying the coherence condition \cite[(7)]{MR4581741} ($(\Phi,\phi)$ is a natural transformation between 2-functors).
Since $\alpha|_\cC = \beta|_\cC = \id_\cC$, we see that $\phi|_\cC$ promotes $\Phi$ to an object of $\cC'$.
We also claim each $\phi_m$ is morphism in $\cC'$.
Indeed,
representing the canonical centralizing structure for objects $c\in \cC$ and $m\in \cM$ by a crossing, we have
$$
\tikzmath{
\draw[thick, blue] (.9,-.7) node[below]{$\scriptstyle \alpha(c)$} -- (.9,.7) -- (.3,1.3) .. controls ++(90:.3cm) and ++(-90:.3cm) .. (-.3,2) node[above]{$\scriptstyle \beta(c)$};
\draw[thick, orange, knot] (.3,-.7) node[below]{$\scriptstyle \alpha(m)$} -- (.3,-.3) -- (-.3,.3) -- (-.3,1.3) .. controls ++(90:.3cm) and ++(-90:.3cm) .. (.3,2) node[above]{$\scriptstyle \beta(m)$};
\draw (-.3,-.7) node[below]{$\scriptstyle \Phi$} -- (-.3,-.3) -- (.3,.3) -- (.3,.7) -- (.9,1.3) -- (.9,2);
\roundNbox{fill=white}{(0,0)}{.3}{.2}{.2}{$\phi_m$}
\roundNbox{fill=white}{(.6,1)}{.3}{.2}{.2}{$\phi_c$}
\draw[very thick, dotted, rounded corners = 5pt] (-.5,.5) rectangle (1.3,1.9);
\node at (-1.5,1.2) {$e_{\beta(m)\boxtimes \Phi,c}$};
}
=
\tikzmath{
\draw[thick, blue] (.9,-.7) node[below]{$\scriptstyle \alpha(c)$} -- (.9,.2) -- (.3,.8) --(.3,1.3) .. controls ++(90:.3cm) and ++(-90:.3cm) .. (-.3,2) node[above]{$\scriptstyle \beta(c)$};
\draw[thick, orange, knot] (.3,-.7) node[below]{$\scriptstyle \alpha(m)$} -- (.3,.2) -- (-.3,.8) -- (-.3,1.3) .. controls ++(90:.3cm) and ++(-90:.3cm) .. (.3,2) node[above]{$\scriptstyle \beta(m)$};
\draw (-.3,-.7) node[below]{$\scriptstyle \Phi$} -- (-.3,.2) -- (.9,.3) -- (.9,2);
\roundNbox{fill=white}{(0,.5)}{.3}{.2}{.8}{$\phi_{m\otimes c}$}
}
=
\tikzmath[scale=-1]{
\draw[thick, blue] (.9,-.7) node[above]{$\scriptstyle \beta(c)$} -- (.9,.2) -- (.3,.8) --(.3,1.3) .. controls ++(90:.3cm) and ++(-90:.3cm) .. (-.3,2) node[below]{$\scriptstyle \alpha(c)$};
\draw[thick, orange, knot] (.3,-.7) node[above]{$\scriptstyle \beta(m)$} -- (.3,.2) -- (-.3,.8) -- (-.3,1.3) .. controls ++(90:.3cm) and ++(-90:.3cm) .. (.3,2) node[below]{$\scriptstyle \alpha(m)$};
\draw (-.3,-.7)  -- (-.3,.2) -- (.9,.3) -- (.9,2) node[below]{$\scriptstyle \Phi$};
\roundNbox{fill=white}{(0,.5)}{.3}{.2}{.8}{$\phi_{c\otimes m}$}
}
=
\tikzmath[scale=-1]{
\draw[thick, blue] (.9,-.7) node[above]{$\scriptstyle \beta(c)$} -- (.9,.7) -- (.3,1.3) .. controls ++(90:.3cm) and ++(-90:.3cm) .. (-.3,2) node[below]{$\scriptstyle \alpha(c)$};
\draw[thick, orange, knot] (.3,-.7) node[above]{$\scriptstyle \beta(m)$} -- (.3,-.3) -- (-.3,.3) -- (-.3,1.3) .. controls ++(90:.3cm) and ++(-90:.3cm) .. (.3,2) node[below]{$\scriptstyle \alpha(m)$};
\draw (-.3,-.7) -- (-.3,-.3) -- (.3,.3) -- (.3,.7) -- (.9,1.3) -- (.9,2) node[below]{$\scriptstyle \Phi$};
\roundNbox{fill=white}{(0,0)}{.3}{.2}{.2}{$\phi_m$}
\roundNbox{fill=white}{(.6,1)}{.3}{.2}{.2}{$\phi_c$}
\draw[very thick, dotted, rounded corners = 5pt] (-.5,.5) rectangle (1.3,1.9);
\node at (-1.5,1.2) {$e_{\Phi\boxtimes \alpha(m),c}$};
}\,.
$$
We conclude that the two fully faithful unitary tensor functors $\cM \hookrightarrow \cC'$ are conjugate by an invertible object in $\cC'$.

Using the equivalence between unitary anchored planar algebras and pointed unitary module tensor categories,
the above argument shows that the composite 
\[
\!\!\Bigg\{\!
\parbox{2.6cm}{
\centerline{\rm Finite depth}\centerline{\rm objects of $\cC'$}
}
\!\Bigg\}\Bigg/\text{\rm conj.}
\,\to\,\,\,
\Bigg\{
\parbox{4.4cm}{
\rm Connected finite depth \centerline{\rm unitary APAs in $Z(\cC)^{\rev}$}
}
\Bigg\}
\!\!\Bigg/\text{\rm iso.}
\,\to\,\,\,
\Bigg\{\!
\parbox{2.6cm}{
\centerline{\rm Finite depth}\centerline{\rm objects of $\cC'$}
}
\!\Bigg\}\Bigg/\text{\rm conj.}
\]
is the identity.

The other composite, from unitary anchored planar algebras to finite depth objects of $\cC'$ back to unitary anchored planar algebras produces an isomorphic unitary anchored planar algebra because unitarily equivalent pointed unitary module tensor categories give equivalent unitary anchored planar algebras by \cite[Thm.~A]{2301.11114}.

\bibliographystyle{amsalpha}
{\footnotesize{
\bibliography{../../bibliography/bibliography}

\providecommand{\bysame}{\leavevmode\hbox to3em{\hrulefill}\thinspace}
\providecommand{\MR}{\relax\ifhmode\unskip\space\fi MR }
\providecommand{\MRhref}[2]{%
  \href{http://www.ams.org/mathscinet-getitem?mr=#1}{#2}
}
\providecommand{\href}[2]{#2}
\begin{thebibliography}{DMNO13}

\bibitem[BBJ19]{MR3975865}
Daniel Barter, Jacob~C. Bridgeman, and Corey Jones, \emph{Domain walls in
  topological phases and the {B}rauer-{P}icard ring for {${\rm Vec}
  (\mathbb{Z}/p\mathbb{Z})$}}, Comm. Math. Phys. \textbf{369} (2019), no.~3,
  1167--1185, \mathscinet{MR3975865} \doi{10.1007/s00220-019-03338-2}.
  \MR{3975865}

\bibitem[BJS21]{MR4228258}
Adrien Brochier, David Jordan, and Noah Snyder, \emph{On dualizability of
  braided tensor categories}, Compos. Math. \textbf{157} (2021), no.~3,
  435--483, \mathscinet{MR4228258} \doi{10.1112/s0010437x20007630}
  \arxiv{1804.07538}. \MR{4228258}

\bibitem[CJP21]{2111.06378}
Quan Chen, Corey Jones, and David Penneys, \emph{A categorical {C}onnes'
  $\chi({M})$}, 2021, \arxiv{2111.06378}.

\bibitem[DMNO13]{MR3039775}
Alexei Davydov, Michael M{\"u}ger, Dmitri Nikshych, and Victor Ostrik,
  \emph{The {W}itt group of non-degenerate braided fusion categories}, J. Reine
  Angew. Math. \textbf{677} (2013), 135--177, \mathscinet{MR3039775}
  \arxiv{1009.2117}. \MR{3039775}

\bibitem[Hau17]{MR3650080}
Rune Haugseng, \emph{The higher {M}orita category of
  {$\mathbb{E}_n$}-algebras}, Geom. Topol. \textbf{21} (2017), no.~3,
  1631--1730, \mathscinet{MR3650080} \doi{10.2140/gt.2017.21.1631}.
  \MR{3650080}

\bibitem[Hen17a]{1701.02052}
Andr\'e Henriques, \emph{Bicommutant categories from conformal nets}, 2017,
  \arxiv{1701.02052}.

\bibitem[Hen17b]{MR3747830}
Andr\'{e}~G. Henriques, \emph{What {C}hern-{S}imons theory assigns to a point},
  Proc. Natl. Acad. Sci. USA \textbf{114} (2017), no.~51, 13418--13423,
  \mathscinet{MR3747830} \doi{10.1073/pnas.1711591114} \arxiv{1503.06254}.
  \MR{3747830}

\bibitem[HP17]{MR3663592}
Andr\'e Henriques and David Penneys, \emph{Bicommutant categories from fusion
  categories}, Selecta Math. (N.S.) \textbf{23} (2017), no.~3, 1669--1708,
  \mathscinet{MR3663592} \doi{10.1007/s00029-016-0251-0} \arxiv{1511.05226}.
  \MR{3663592}

\bibitem[HP23]{MR4581741}
Andr\'{e} Henriques and David Penneys, \emph{Representations of fusion
  categories and their commutants}, Selecta Math. (N.S.) \textbf{29} (2023),
  no.~3, Paper No. 38, \mathscinet{MR4581741} \doi{10.1007/s00029-023-00841-2}
  \arxiv{2004.08271}. \MR{4581741}

\bibitem[HPT16]{MR3578212}
Andr\'e Henriques, David Penneys, and James Tener, \emph{Categorified trace for
  module tensor categories over braided tensor categories}, Doc. Math.
  \textbf{21} (2016), 1089--1149, \mathscinet{MR3578212} \arxiv{1509.02937}.
  \MR{3578212}

\bibitem[HPT23a]{MR4528312}
Andr\'{e} Henriques, David Penneys, and James Tener, \emph{Planar {A}lgebras in
  {B}raided {T}ensor {C}ategories}, Mem. Amer. Math. Soc. \textbf{282} (2023),
  no.~1392, \mathscinet{MR4528312} \doi{10.1090/memo/1392} \arxiv{1607.06041}.
  \MR{4528312}

\bibitem[HPT23b]{2301.11114}
Andr\'e Henriques, David Penneys, and James Tener, \emph{Unitary anchored
  planar algebras}, 2023, \arxiv{2301.11114}.

\bibitem[Izu00]{MR1782145}
Masaki Izumi, \emph{The structure of sectors associated with {L}ongo-{R}ehren
  inclusions. {I}. {G}eneral theory}, Comm. Math. Phys. \textbf{213} (2000),
  no.~1, 127--179, \mathscinet{MR1782145} \doi{10.1007/s002200000234}.
  \MR{1782145}

\bibitem[Izu17]{MR3635673}
\bysame, \emph{A {C}untz algebra approach to the classification of near-group
  categories}, Proceedings of the 2014 {M}aui and 2015 {Q}inhuangdao
  conferences in honour of {V}aughan {F}. {R}. {J}ones' 60th birthday, Proc.
  Centre Math. Appl. Austral. Nat. Univ., vol.~46, Austral. Nat. Univ.,
  Canberra, 2017, \mathscinet{MR3635673} \arxiv{1512.04288}, pp.~222--343.
  \MR{3635673}

\bibitem[JFS17]{MR3590516}
Theo Johnson-Freyd and Claudia Scheimbauer, \emph{({O}p)lax natural
  transformations, twisted quantum field theories, and ``even higher'' {M}orita
  categories}, Adv. Math. \textbf{307} (2017), 147--223, \mathscinet{MR3590516}
  \doi{10.1016/j.aim.2016.11.014} \arxiv{1502.06526}. \MR{3590516}

\bibitem[JMPP22]{MR4498161}
Corey Jones, Scott Morrison, David Penneys, and Julia Plavnik, \emph{Extension
  theory for braided-enriched fusion categories}, Int. Math. Res. Not. IMRN
  (2022), no.~20, 15632--15683, \mathscinet{MR4498161}
  \doi{10.1093/imrn/rnab133} \arxiv{1910.03178}. \MR{4498161}

\bibitem[JP17]{MR3687214}
Corey Jones and David Penneys, \emph{Operator algebras in rigid {$\rm
  C^*$}-tensor categories}, Comm. Math. Phys. \textbf{355} (2017), no.~3,
  1121--1188, \mathscinet{MR3687214} \doi{10.1007/s00220-017-2964-0}
  \arxiv{1611.04620}. \MR{3687214}

\bibitem[LR97]{MR1444286}
R.~Longo and J.~E. Roberts, \emph{A theory of dimension}, $K$-Theory
  \textbf{11} (1997), no.~2, 103--159, \mathscinet{MR1444286}
  \doi{10.1023/A:1007714415067} \arxiv{funct-an/9604008}. \MR{1444286}

\bibitem[Lur17]{LurieHigherAlgebra}
Jacob Lurie, \emph{Higher {A}lgebra}, 2017, Available at
  \url{https://www.math.ias.edu/~lurie/papers/HA.pdf}.

\bibitem[M{\"u}g03]{MR1966525}
Michael M{\"u}ger, \emph{From subfactors to categories and topology. {II}.
  {T}he quantum double of tensor categories and subfactors}, J. Pure Appl.
  Algebra \textbf{180} (2003), no.~1-2, 159--219, \mathscinet{MR1966525}
  \doi{10.1016/S0022-4049(02)00248-7} \arXiv{math.CT/0111205}.

\bibitem[MW12]{MR2978449}
Scott Morrison and Kevin Walker, \emph{Blob homology}, Geom. Topol. \textbf{16}
  (2012), no.~3, 1481--1607, \mathscinet{MR2978449}
  \doi{10.2140/gt.2012.16.1481}. \MR{2978449}

\bibitem[Pen20]{MR4133163}
David Penneys, \emph{Unitary dual functors for unitary multitensor categories},
  High. Struct. \textbf{4} (2020), no.~2, 22--56, \mathscinet{MR4133163}
  \arxiv{1808.00323}. \MR{4133163}

\bibitem[Pop90]{MR1055708}
Sorin Popa, \emph{Classification of subfactors: the reduction to commuting
  squares}, Invent. Math. \textbf{101} (1990), no.~1, 19--43,
  \mathscinet{MR1055708}, \doi{10.1007/BF01231494}.

\bibitem[Pop95]{MR1339767}
\bysame, \emph{Classification of subfactors and their endomorphisms}, CBMS
  Regional Conference Series in Mathematics, vol.~86, Published for the
  Conference Board of the Mathematical Sciences, Washington, DC, 1995,
  \mathscinet{MR1339767}. \MR{1339767 (96d:46085)}

\bibitem[Row19]{mitchell-thesis}
Mitchell Rowett, \emph{Monoidal ladder categories}, 2019, available at
  \url{https://tqft.net/web/research/students/MitchellRowett/thesis.pdf}.

\bibitem[Tom21]{MR4236062}
Reiji Tomatsu, \emph{Centrally {F}ree {A}ctions of {A}menable {$\rm
  C^*$}-tensor categories on von {N}eumann algebras}, Comm. Math. Phys.
  \textbf{383} (2021), no.~1, 71--152, \mathscinet{MR4236062}
  \doi{10.1007/s00220-021-04037-7} \arxiv{1812.04222}. \MR{4236062}

\end{thebibliography}
}}
\end{document}